\UseRawInputEncoding
\documentclass[onefignum,onetabnum]{siamart190516}

\usepackage{lipsum,bm}
\usepackage{amsfonts,wasysym}
\usepackage{graphicx}
\usepackage{epstopdf}
\usepackage{algorithmic}
\usepackage{algorithm}
\usepackage{amsopn}
\usepackage{mathtools}
\usepackage{amssymb}
\usepackage{makecell}
\ifpdf
\DeclareGraphicsExtensions{.eps,.pdf,.png,.jpg}
\else
\DeclareGraphicsExtensions{.eps}
\fi

\newsiamremark{rem}{Remark}
\newsiamremark{defn}{Definition}
\newsiamthm{cor}{Corollary}
\newsiamthm{assum}{Assumption}
\newsiamthm{thm}{Theorem}
\newsiamthm{prop}{Proposition}
\newsiamthm{lem}{Lemma}
\newsiamthm{example}{Example}

\title{Convergence Analysis of a Quasi-Monte Carlo-Based Deep Learning Algorithm for Solving Partial Differential Equations}

\author{Fengjiang Fu\thanks{Department of Mathematical Sciences, Tsinghua University, Beijing, 100084, People's Republic of China (\email{ffj18@mails.tsinghua.edu.cn}).}
	\and Xiaoqun Wang\thanks{Department of Mathematical Sciences, Tsinghua University, Beijing 100084, People's Republic of China (\email{wangxiaoqun@mail.tsinghua.edu.cn}).}}

\ifpdf
\hypersetup{
	pdftitle={Convergence Analysis of a Quasi-Monte Carlo-Based Deep Learning Algorithm for Solving Partial Differential Equations},
	pdfauthor={Fengjiang Fu and Xiaoqun Wang}
}
\fi

\begin{document}
	
	\maketitle
	
	% REQUIRED
	\begin{abstract}
		Deep learning methods have achieved great success in solving partial differential equations~(PDEs), where the loss is often defined as an integral. The accuracy and efficiency of these algorithms depend greatly on the quadrature method. We propose to apply quasi-Monte Carlo~(QMC) methods to the Deep Ritz Method~(DRM) for solving the Neumann problems for the Poisson equation and the static Schr\"{o}dinger equation. For error estimation, we decompose the error of using the deep learning algorithm to solve PDEs into the generalization error, the approximation error and the training error. We establish the upper bounds and prove that QMC-based DRM achieves an asymptotically smaller error bound than DRM. Numerical experiments show that the proposed method converges faster in all cases and the variances of the gradient estimators of randomized QMC-based DRM are much smaller than those of DRM, which illustrates the superiority of QMC in deep learning over MC.
	\end{abstract}
	
	% REQUIRED
	
	% REQUIRED
	
\section{Introduction}
\label{sec1}
Partial differential equations~(PDEs) are classical models for describing problems arising in physics, finance and engineering. Solving PDEs by deep learning has attracted considerable attention, see~\cite{E1, Han, Lagaris}. Recently, a variety of well-designed deep learning algorithms for solving PDEs have been proposed, such as the physics-informed neural networks~(PINNs)~\cite{Raissi}, the Deep Ritz Method~(DRM)~\cite{E2} and the Deep Galerkin Method~(DGM)~\cite{Sirignano}. The basic idea of these algorithms is to minimize the loss by training the deep neural network. These deep learning algorithms have shown satisfactory efficiency and a wide range of application scenarios. However, researchers are not satisfied with treating a deep learning algorithm as a black box. It is desirable to identify the factors that drive the algorithm in mathematics and improve the algorithm by modifying these factors.

In this paper, we study the effect of the different sampling strategies on DRM. There have been some papers about the error analysis of DRM, see~\cite{Duan, Jiao, Lu2}. Briefly, the essence of DRM is to solve
\begin{equation*}
 \min_{u \in H^{1}\left(\Omega\right)}I(u),
\end{equation*}
where $I(u)$ is in the form of an integral and $H^{1}\left(\Omega\right)$ is a Sobolev space. Obviously, the quadrature method plays an important role in DRM. For high-dimensional PDEs, the algorithm may suffer from the curse of dimensionality. We aim to enhance the accuracy and efficiency of DRM by combining it with a new sampling strategy. To be specific, the accuracy is expressed in terms of the total error, namely the difference between the limit of the algorithm output and the exact solution of the PDE, and the efficiency is measured by the convergence rate and stability of the algorithm. 

Quasi-Monte Carlo~(QMC) methods are efficient quadrature methods, which choose deterministic points, rather than random points, as sample points. QMC methods are widely used in finance~\cite{LEcuyer}, statistics~\cite{Fang}, etc. The Koksma-Hlawka inequality~\cite{Niederreiter} yields that QMC integration has an error bound in the order of $O\left(n^{-1}(\log n)^{d}\right)$ for the integrands with suitable smoothness, where $n$ is the sample size and $d$ is the dimension of the domain of the integrand. It is easy to see that the order of QMC is asymptotically better than that of Monte Carlo~(MC). Although the error bound of QMC depends on the dimension, there have been many results that indicate the superiority of QMC over MC in high dimension~\cite{Sobol2}. Furthermore, Caflisch et al.~\cite{Caflisch} and Wang et al.~\cite{Wang} attribute the superiority of QMC to the effective dimension of the integrand, which is usually much lower than the nominal dimension. We believe that the integrands arising in deep learning outlined in this paper have similar characteristics. 

Recently, the application of QMC methods combined with finite element methods to solve some classes of PDEs with random coefficients has achieved good performance~\cite{Kuo} and some researchers have applied QMC methods to machine learning successfully~\cite{Dick1, Longo, Lyu, Mishra}. We propose to combine QMC methods with DRM (abbreviated as DRM-QMC) for solving the Poisson equation and the static Schr\"{o}dinger equation equipped with the Neumann boundary condition. DRM-QMC can achieve asymptotically smaller error bound than DRM. The proposed algorithm converges faster and is more stable than DRM. To prove these results, we will
\begin{enumerate}
  \item[(i)] formalize DRM-QMC through training the deep neural network by low discrepancy sequences,
  \item[(ii)] decompose the total error into three parts, which correspond to the generalization error, the approximation error and the training error, and then establish their upper bounds to demonstrate their relationship with the mini-batch size, and
  \item[(iii)] implement the algorithms by performing numerical experiments and compare the training processes with respect to different sampling strategies and mini-batch sizes.
\end{enumerate}

This paper is organized as follows. The proposed algorithm is presented in Section~\ref{sec2}. The error bounds of DRM and DRM-QMC are analyzed in Section~\ref{sec3}, and the superiority of DRM-QMC is verified by the numerical experiments in Section~\ref{sec4}. Section~\ref{sec5} concludes the paper.

\section{QMC methods in DRM}
\label{sec2} 
DRM is an algorithm of solving extremum problems in a deep neural network as substitute for computing the numerical solution of some PDE problems, where the losses are based on the variational form of PDEs~\cite{E2}. To study the effect of different sampling strategies on DRM, we first introduce two PDE problems considered in this paper and the DRM based on random sampling strategy. Next, we introduce QMC methods and propose to combine QMC methods with DRM, leading to the DRM-QMC algorithm.

  \subsection{PDE problems and the corresponding variational problems}
  Let $\Omega=[0,1]^{d}$ denote the unit hypercube in $\mathbb{R}^{d}$ and $\partial \Omega$ be its boundary. It should be noted that the results in this paper can be generalized for any bounded rectangular domains. Without loss of generality, we study the following two prototype elliptic PDEs on $\Omega$ equipped with the Neumann boundary condition, which play crucial roles in the study of electrostatics, mechanical engineering and quantum mechanics. One is the Poisson equation
  \begin{equation} \label{eq1}
  \left\{
    \begin{aligned}
      -\Delta u&=f, \  \rm{in} \ \Omega, \\
      \frac{\partial u}{\partial \boldsymbol{n}}&=0, \ \rm{on} \ \partial \Omega,
    \end{aligned}
  \right.
  \end{equation}
  where $\Delta$ is the Laplace operator and $\boldsymbol{n}$ is the unit outward normal vector. Another is the static Schr\"{o}dinger equation
  \begin{equation} \label{eq2}
  \left\{
    \begin{aligned}
      -\Delta u+Vu&=g, \ \rm{in} \ \Omega, \\
      \frac{\partial u}{\partial \boldsymbol{n}}&=0, \ \rm{on} \ \partial \Omega.
    \end{aligned}
  \right.
  \end{equation}
  Assume that $f\in L^{2}(\Omega)$ with $\int_{\Omega} f(x) \mathrm{d}x=0$, $V\in L^{\infty}(\Omega)$ and $0<V_{min}\leqslant V(x)\leqslant V_{max}< \infty$ in $\Omega$, where $V_{min}$ and $V_{max}$ are two constants. In the following sections, we will make more assumptions. 
  
  To start with, we introduce the results in~\cite{Lu2} which state that~(\ref{eq1}) and~(\ref{eq2}) can be turned into variational problems, and the difference between the unique weak solution and any element in the Sobolev space $H^{1}(\Omega)$ can be bounded by the difference between their loss functional values.
  \begin{theorem}\label{thm21}
    \begin{enumerate}
      \item[(i)] There exists a unique weak solution $u_{P}^{*}$ to the Poisson equation~(\ref{eq1}) with $\int\nolimits_{\Omega} u_{P}^{*}\,\mathrm{d}x=0$. Moreover, $u_{P}^{*}$ satisfies
            \begin{equation*}
              u_{P}^{*}=\mathop{\arg\min}_{u \in H^{1}\left(\Omega\right)}\mathcal{L}_{P}(u),
            \end{equation*}
            where
            \begin{small}
            \begin{equation*}
              \mathcal{L}_{P}(u)\coloneqq\int\nolimits_{\Omega}\left(\frac{1}{2}\left\lVert \triangledown u(x)\right\rVert_{\ell_{2}}^{2} -f(x)u(x)\right)\mathrm{d}x+\frac{1}{2}\left( \int\nolimits_{\Omega}u(x)\mathrm{d}x\right)^{2}
            \end{equation*}
            denotes the loss functional with respect to~(\ref{eq1}). For any $u \in H^{1}\left(\Omega\right)$, it holds
            \begin{equation*}
            2\left(\mathcal{L}_{P}(u)-\mathcal{L}_{P}(u_{P}^*)\right)\leqslant \left\lVert u-u_{P}^*\right\rVert_{H^{1}(\Omega)}^{2}\leqslant 2 \max \left\{2C_{Pc}+1,2\right\}\left(\mathcal{L}_{P}(u)-\mathcal{L}_{P}(u_{P}^*)\right),
            \end{equation*}
            \end{small}
            where $C_{Pc}$ is the $Poincar\acute{e}$ constant on $\Omega$.
      \item[(ii)] If $g\in L^{\infty}(\Omega)$, then there exists a unique weak solution $u_{S}^{*}$ to the static Schr\"{o}dinger equation~(\ref{eq2}). Similarly, $u_{S}^{*}$ satisfies
            \begin{equation*}
              u_{S}^*=\mathop{\arg\min}_{u \in H^{1}\left(\Omega\right)}\mathcal{L}_{S}(u),
            \end{equation*}
            where
            \begin{equation*}
            \mathcal{L}_{S}(u)\coloneqq\int\nolimits_{\Omega} \left(\frac{1}{2}\left\lVert \triangledown u(x)\right\rVert_{\ell_{2}}^{2}+\frac{1}{2}V(x)\left\lvert u(x)\right\rvert^{2}-g(x)u(x) \right)\mathrm{d}x
            \end{equation*}
            denotes the loss functional with respect to~(\ref{eq2}). For any $u \in H^{1}\left(\Omega\right)$, it holds
            \begin{equation*}
              \frac{2\left(\mathcal{L}_{S}(u)-\mathcal{L}_{S}(u_{S}^* )\right)}{\max \left\{1,V_{max}\right\}} \leqslant \left\lVert u-u_{S}^* \right\rVert_{H^{1}(\Omega)}^{2}\leqslant \frac{2\left(\mathcal{L}_{S}(u)-\mathcal{L}_{S}(u_{S}^*)\right)}{\min \left\{1,V_{min}\right\}}.
            \end{equation*}
    \end{enumerate}
  \end{theorem}

  \begin{rem}
    The loss functional $\mathcal{L}_{P}(u)$ can be replaced by a more intuitive one
    \begin{equation}\label{eq3}
      \int\nolimits_{\Omega}\left(\frac{1}{2}\left\lVert \triangledown u(x)\right\rVert_{\ell_{2}}^{2} -f(x)u(x)\right)\mathrm{d}x.
    \end{equation}
    Due to the fact that the unique weak solution of~(\ref{eq1}) satisfies $\int\nolimits_{\Omega} u_{P}^*\,\mathrm{d}x=0$, the loss functional~(\ref{eq3}) is not suitable for training a neural network solution. Lu et al.~\cite{Lu2} proved that these two loss functionals get the same solution.
  \end{rem}

  \subsection{Basic ideas of DRM}
  The DRM is an algorithm to find the optimal approximation of the solution of PDEs in a deep neural network. The basic ideas of the DRM are as follows~\cite{E2}:
  \begin{enumerate}
    \item[(i)] Deep neural network-based approximation of the trial function.
    \item[(ii)] Computation of the loss functional based on MC method.
    \item[(iii)] The optimization algorithm for finding the optimal weight parameters.
  \end{enumerate}

  First, we construct the deep neural network by a prevalent method. Every realization of this deep neural network is a composition of functions. For the feasibility of error analysis, we only consider the deep neural network with fixed depth, width and parameters bound.
  \begin{definition}
    For a given activation function $\sigma(\cdot)$, a deep neural network of depth $L$ is a function class defined as
    \begin{equation*}
      \mathcal{F}_{L,W,B_{\Theta}}\coloneqq \left\{v(\cdot;\theta):v(\cdot;\theta)=T_{L}\circ \sigma \circ T_{L-1} \circ \sigma \circ \dots \circ \sigma \circ T_{1}(\cdot)\right\},
    \end{equation*}
    where $T_{i}(s)=A_{i}s+B_{i} \left(A_{i} \in \mathbb{R}^{d_{i}\times d_{i-1}}, B_{i} \in \mathbb{R}^{d_{i}}, i=1,2,\dots,L\right)$ are affine transformations and $W=\max(d_{1},d_{2},\dots,d_{L})$ denotes the given width of this deep neural network. Moreover, $\theta =\left\{\left(A_{i},B_{i}\right) \right\}_{i=1}^{L}$ are called the weight parameters or weights of $v(x;\theta)$. Let $\theta \in \Theta \coloneqq \left\{\theta \in \mathbb{R}^{D}:\left\lVert \theta\right\rVert_{\ell_{\infty}} \leqslant B_{\Theta} \right\}$, where $D=\sum_{i= 1}^{L}d_{i}+ \sum_{i= 1}^{L}d_{i}\times d_{i-1}$ and $B_{\Theta}$ is a given positive constant.
  \end{definition}

  \begin{rem}\label{remark22}
    Without loss of generality, we consider $d_{L}=1$ in this paper. By vectorizing the matrix and combining multiple column vectors into one column vector, we can write $\theta \in \mathbb{R}^{D}$.
    For ease of notation, we denote the input dimension by $d$ and the deep neural network by $\mathcal{F}$. Furthermore, we choose the swish function
    \begin{equation*}
      \sigma(x)=\frac{x}{1+\exp(-x)}
    \end{equation*}
    as the activation function, which guarantees the smoothness of the functions in $\mathcal{F}$. There is one-to-one correspondence between the function in $\mathcal{F}$ and its weight parameters $\theta$.
  \end{rem}

  In deep learning, we treat everything as a vector. Each layer of a deep learning model performs a simple geometric transformation on the data passing through it, and these transformations are then parameterized by the weights. For example, in the DRM, we see the deep neural network as a weight vector space and use the realizations corresponding to these weight vectors to approximate the solution of PDEs. By converting the problem of solving elliptic PDEs into a variational problem, we can find the optimal weight parameters by a gradient descent method, which corresponds to the optimal approximation of the solution of PDEs. 

  In applications, we must compute the loss functional by a numerical quadrature method. Based on MC method, we define the empirical loss functionals with respect to~(\ref{eq1}) and~(\ref{eq2}) as
  \begin{equation*}
    \mathcal{L}_{n,P}(u)=\frac{1}{n}\sum_{j = 1}^{n}\left(\frac{1}{2}\left\lVert \triangledown_{x} u(X_{j})\right\rVert_{\ell_{2}}^{2} - f(X_{j})u(X_{j})\right) + \frac{1}{2}\left(\frac{1}{n}\sum_{j = 1}^{n}u(X_{j})\right)^{2}
  \end{equation*}
  and
  \begin{equation*}
    \mathcal{L}_{n,S}(u)=\frac{1}{n}\sum_{j = 1}^{n}\left(\frac{1}{2}\left\lVert \triangledown_{x} u(X_{j})\right\rVert_{\ell_{2}}^{2} - g(X_{j})u(X_{j}) + \frac{1}{2}V(X_{j})\left\lvert u(X_{j})\right\rvert^{2} \right),
  \end{equation*}
  where $n$ is called mini-batch size and $\left\{X_{j}\right\}_{j=1}^{n}$ are sample points. For DRM, $\left\{X_{j}\right\}_{j=1}^{n}$ is a set of independent and identically distributed~(i.i.d.) uniform random points on $\Omega$~\cite{E2}.

  To avoid confusion, we emphasize that $\triangledown_{x} v(x;\theta)$ refers to the partial derivative of $v(x;\theta)$ with respect to $x$. Later we will use the notation $\triangledown_{\theta} v(x;\theta)$ as the partial derivative of $v(x;\theta)$ with respect to $\theta$.

  Next, it should be noted that the optimal approximation of the solution of PDEs in DRM is equivalent to the minimizer of the empirical loss functional. Taking the variational problem associated with~(\ref{eq2}) as an example, we define
  \begin{equation*}
    \mu_{S} (X;\theta) \coloneqq \frac{1}{2}\left\lVert \triangledown_{x} v(X;\theta)\right\rVert_{\ell_{2}}^{2} - g(X)v(X;\theta) + \frac{1}{2}V(X)\left\lvert v(X;\theta)\right\rvert^{2}.
  \end{equation*}
  Then we have
  \begin{equation*}
    \mathcal{L}_{S}(v\left(\cdot; \theta \right)) \thickapprox \mathcal{L}_{n,S}(v\left(\cdot; \theta \right) )= \frac{1}{n}\sum_{j = 1}^{n} \mu_{S} (X_{j};\theta).
  \end{equation*}
  Hence, the solution of~(\ref{eq2}) can be approximated by the solution of
  \begin{equation*}
    \min_{v \in \mathcal{F}}\mathcal{L}_{n,S}(v).
  \end{equation*}
  After recognizing that the initial problem can be replaced by an extremum problem where only a finite number of parameters are needed to be determined, we use the mini-batch gradient descent method to find the optimal weight parameters. The realization corresponding to the optimal weight parameters is the approximation to the solution of PDEs.

  \subsection{DRM-QMC}
  It is clear that the key in DRM is the efficiency of the quadrature method. We intend to replace the uniformly distributed sample points in DRM by low discrepancy points~\cite{Chen}, leading to the DRM-QMC. Before that, we introduce QMC briefly. 

  Consider the approximation of an integral 
  \begin{equation*}
    I =\int\nolimits_{[0,1]^{d}}\mu(x) \mathrm{d}x.
  \end{equation*}
  For MC, we first generate i.i.d. random points $X_{1},X_{2},\dots X_{n}$ from the uniform distribution on $\left[0,1\right]^{d}$. By evaluating the function $\mu(x)$ at these random points and averaging the function values, we obtain the MC estimate
  \begin{equation} \label{eq4}
    \hat{I}=\frac{1}{n}\sum_{j = 1}^{n}\mu(X_{j}).
  \end{equation}

  The idea of QMC methods is to choose deterministic points with better uniformity to replace random points and estimate $I$ by the same form of~(\ref{eq4}). The deterministic points in QMC are called low discrepancy points. We introduce several fundamental concepts in QMC.
  \begin{definition}
    An infinite sequence $\left\{\boldsymbol{x}_{1},\boldsymbol{x}_{2},\dots\right\} \subset [0,1]^{d}$ is called a low discrepancy sequence if the star discrepancy of its first $n$ points satisfies
    \begin{equation*}
      D_{n}^{*}(\boldsymbol{x}_{1},\dots,\boldsymbol{x}_{n}) = O(n^{-1}\left(\log n\right)^{d}).
    \end{equation*}
  \end{definition}
  The definition of star discrepancy can be found in~\cite{Niederreiter}. For ease of notation, we denote $D_{n}^{*}$ as the star discrepancy of $\boldsymbol{x}_{1},\dots,\boldsymbol{x}_{n}$ when it is clear which points are used.

  There are various constructions of low discrepancy sequences, such as Halton, Faure, Sobol' and Niederreiter sequences as well as others~\cite{Dick2}. The integration error is bounded by the Koksma-Hlawka inequality~\cite{Niederreiter}. Before we state the Koksma-Hlawka inequality, we introduce the variation in the sense of Hardy and Krause for smooth functions.
  \begin{definition}
    For a function $\mu$ with continuous mixed partial derivatives of up to order $d$ over $\left[0,1\right]^{d}$, we define
    \begin{equation*}
      V^{(k)}(\mu;i_{1},\dots,i_{k})\coloneqq \int_{0}^{1}\dots \int_{0}^{1} \left\lvert \frac{\partial^{k}\mu(\boldsymbol{u}_{\boldsymbol{i}};\boldsymbol{1}_{\boldsymbol{-i}})}{\partial u_{i_{1}}\dots\partial u_{i_{k}}} \right\rvert  \,\mathrm{d}u_{i_{1}}\dots \mathrm{d}u_{i_{k}},
    \end{equation*}
    where $(\boldsymbol{u}_{\boldsymbol{i}};\boldsymbol{1}_{\boldsymbol{-i}})$ refers to the point whose $j$th component is $u_{j}$ if $j \in \boldsymbol{i}=\left\{{i_{1},\dots,i_{k}}\right\}$ and 1 otherwise. Then the variation of $\mu$ in the sense of Hardy and Krause is defined as
    \begin{equation*}
      V_{HK}(\mu)\coloneqq \sum_{k = 1}^{d} \sum_{1\leqslant i_{1}<\dots<i_{k}\leqslant d}V^{(k)}(\mu;i_{1},\dots,i_{k}).
    \end{equation*}
  \end{definition}
  Due to the smoothness of functions in $\mathcal{F}$, only the definition of the Hardy-Krause variation for smooth functions is given here. The definition of the Hardy-Krause variation for more general functions can be found in~\cite{Niederreiter}.
  \begin{proposition}[Koksma-Hlawka inequality]
    If the function $\mu$ has bounded Hardy-Krause variation $V_{HK}(\mu)$, then for any $\boldsymbol{x}_{1},\dots,\boldsymbol{x}_{n} \in \left[0,1\right)^{d}$, there holds
    \begin{equation*}
      \left\lvert \frac{1}{n}\sum_{j = 1}^{n}\mu(\boldsymbol{x}_{j})- \int_{\left[0,1\right]^{d}} \mu(x) \,\mathrm{d}x \right\rvert \leqslant V_{HK}(\mu)D_{n}^{*}(\boldsymbol{x}_{1},\dots,\boldsymbol{x}_{n}).
    \end{equation*}
  \end{proposition}
  Moreover, for a vector-valued function $\mu(x)=\left(\mu_{1}(x),\dots,\mu_{m}(x)\right)^{T}$, let
  \begin{equation*}
    V_{HK}(\mu)\coloneqq \sum_{i = 1}^{m}V_{HK}(\mu_{i}),  
  \end{equation*}
  and the Koksma-Hlawka inequality still holds in $\left\lVert \cdot\right\rVert_{\ell_{2}}$-norm.
  
  Sobol' sequences~\cite{Sobol1} are widely used (t,d)-sequences in base 2, and the definition of digital sequence can be found in~\cite{Dick2}. That is, by taking mini-batch size $n=2^{\tau},\tau \in \mathbb{N}^{+}$, better equidistribution may be obtained. Taking the static Schr\"{o}dinger equation~(\ref{eq2}) for example, we present the DRM-QMC in Algorithm~\ref{alg1}. In practice, the stepsize $\alpha_{k}$ in Algorithm~\ref{alg1} comes from \textit{Adam}~\cite{Kingma}, which provides an algorithm for first-order gradient-based optimization of stochastic objective functions.
  \begin{algorithm}[htbp]
    \caption{Quasi-Monte Carlo-based Deep Ritz Method}
    \label{alg1}
    \renewcommand{\algorithmicrequire}{\textbf{Input:}}
    \renewcommand{\algorithmicensure}{\textbf{Output:}}
    
    \begin{algorithmic}[1]
        \REQUIRE Initial parameters $\theta_{0}$, mini-batch size $n=2^{\tau}, \tau \in \mathbb{N}^{+}$ and iteration number $T$  %%input
        \ENSURE Parameters after T iterations $\theta_{T}$    %%output
        
        \STATE  Generate Sobol' sequence, denoted by $\left\{P_{j}\right\}$.
        
        \FOR{$k=0,1,2,\dots,T-1$}
            \STATE set $X_{j,k}=P_{k 2^{\tau}+j}$, for $j=1,\dots,2^{\tau}$,
            \STATE $\theta_{k+1}=\theta_{k}-\alpha_{k} \triangledown_{\theta}\left[2^{-\tau }\sum_{j = 1}^{2^{\tau }}\mu_{S} (X_{j,k};\theta_{k}) \right]$.
        \ENDFOR
        
        \RETURN Outputs
    \end{algorithmic}
  \end{algorithm}

  Now we compare MC and QMC for integration roughly. For MC, by the central limit theorem, the root mean squared error~(RMSE) of MC estimate is $O(n^{-1/2})$. For QMC, the error bound is of order $O(n^{-1}\left(\log n\right)^{d})$ based on the Koksma-Hlawka inequality. For a fixed dimension, QMC asymptotically converges faster than MC. Numerical experiments in various applications demonstrate that QMC usually performs better than MC, see~\cite{Paskov, Sobol3}. To clarify the improvement of QMC on the DRM, we will analyze the error bounds of DRM with different sampling strategies in Section~\ref{sec3} and compare the training processes in Section~\ref{sec4}.
  
\section{Error Analysis} \label{sec3}
  To derive the error bound of using the deep learning algorithms to solve PDEs, we decompose the total error into three parts:
  \begin{enumerate}
    \item[(i)] Generalization error: the error of the approximate solution on predicting unseen data.
    \item[(ii)] Approximation error: the error of approximating the solution of the PDEs using neural networks.
    \item[(iii)] Training error: the error caused by the optimization algorithm used in the training process.
  \end{enumerate}

  We give the mathematical formulations of these errors. For ease of notation, we unify the loss functional and the empirical loss functional of~(\ref{eq1}) and~(\ref{eq2}) as $\mathcal{L}$ and $\mathcal{L}_{n}$. Namely, we omit the subscript $P$ or $S$ when it is clear from the context. We define
  \begin{align*}
    u^{*} & \coloneqq \mathop{\arg\min}_{u \in H^{1}\left(\Omega\right)}\mathcal{L}(u), & u_{\mathcal{F}} & \coloneqq \mathop{\arg\min}_{v \in \mathcal{F}}\mathcal{L}(v), \\
    u_{n} & \coloneqq \mathop{\arg\min}_{v \in \mathcal{F}}\mathcal{L}_{n}(v),          & u^{(k)}           & \coloneqq v(\cdot ; \theta_{k}) \in \mathcal{F},
  \end{align*}
  where $H^{1}\left(\Omega\right)$ is the Sobolev space and $\mathcal{F}$ is the deep neural network. From Theorem~\ref{thm21}, the difference between the output after $k$ iterations $u^{(k)}$ and the unique weak solution $u^{*}$ is bounded by the difference between their loss functional values. Hence, we study the upper bound on
  \begin{equation*}
    \Delta \mathcal{L}_{k} \coloneqq \mathcal{L}(u^{(k)})-\mathcal{L}(u^{*}).
  \end{equation*}

  From the definition, we know that $\Delta \mathcal{L}_{k}$ must be non-negative. Now we decompose it into three parts corresponding to the generalization error, the approximation error and the training error as follows
    \begin{small}
    \begin{equation*}
      \Delta \mathcal{L}_{k} = \mathcal{L}(u^{(k)})-\mathcal{L}(u_{n})+\mathcal{L}(u_{n})-\mathcal{L}_{n}(u_{n})+\mathcal{L}_{n}(u_{n})-\mathcal{L}_{n}(u_{\mathcal{F}})+\mathcal{L}_{n}(u_{\mathcal{F}})-\mathcal{L}(u_{\mathcal{F}})+\mathcal{L}(u_{\mathcal{F}})-\mathcal{L}(u^{*}).
    \end{equation*}
    \end{small}
  The definition of $u_{n}$ leads to $\mathcal{L}_{n}(u_{n})-\mathcal{L}_{n}(u_{\mathcal{F}}) \leqslant 0$. Exchanging summation order, we have
  \begin{eqnarray*}
    \Delta \mathcal{L}_{k} & \leqslant & \mathcal{L}(u^{(k)})-\mathcal{L}(u_{\mathcal{F}})+\mathcal{L}(u_{n})-\mathcal{L}_{n}(u_{n})+\mathcal{L}_{n}(u_{\mathcal{F}})-\mathcal{L}(u_{n})+\mathcal{L}(u_{\mathcal{F}})-\mathcal{L}(u^{*}) \\
    & \leqslant & \underbrace{\mathcal{L}(u^{(k)})-\mathcal{L}(u_{\mathcal{F}})}_{training \ error}+\underbrace{\mathcal{L}(u_{n})-\mathcal{L}_{n}(u_{n})+\mathcal{L}_{n}(u_{\mathcal{F}})-\mathcal{L}(u_{\mathcal{F}})}_{generalization \ error}+\underbrace{\mathcal{L}(u_{\mathcal{F}})-\mathcal{L}(u^{*})}_{approximation \ error},
  \end{eqnarray*}
  where the second inequality follows from
  \begin{eqnarray*}
    \mathcal{L}_{n}(u_{\mathcal{F}})-\mathcal{L}(u_{n})  = \mathcal{L}_{n}(u_{\mathcal{F}})-\mathcal{L}(u_{\mathcal{F}})+\mathcal{L}(u_{\mathcal{F}})-\mathcal{L}(u_{n}) \leqslant  \mathcal{L}_{n}(u_{\mathcal{F}})-\mathcal{L}(u_{\mathcal{F}}).
  \end{eqnarray*}

  We usually decompose the generalization error into $\mathcal{L}(u_{n})-\mathcal{L}_{n}(u_{n})$ and $\mathcal{L}_{n}(u_{\mathcal{F}})-\mathcal{L}(u_{\mathcal{F}})$ in theoretical analysis. In the actual applications, considering the upper bound on $\Delta \mathcal{L}_{k}$ for a fixed number of iterations $k$ is infeasible due to the randomness of the starting point selection and the unpredictable complexity of the PDE problems. Instead, we are interested in the minimal error achieved by the algorithm. Hence, we take the upper bound on $\lim_{k\to\infty}\Delta \mathcal{L}_{k}$ as a criterion for accuracy. Furthermore, since the sample points in MC are random, it is reasonable to consider the mathematical expectation of $\lim_{k\to\infty}\Delta \mathcal{L}_{k}$. For DRM, we define
  \begin{align*}
    \Delta \mathcal{L}_{MCgen1} & \coloneqq \mathbb{E}\left[\left\lvert \mathcal{L}(u_{n})-\mathcal{L}_{n}(u_{n}) \right\rvert \right], & \Delta \mathcal{L}_{MCgen2} & \coloneqq  \mathbb{E}\left[\left\lvert \mathcal{L}(u_{\mathcal{F}})-\mathcal{L}_{n}(u_{\mathcal{F}}) \right\rvert \right],     \\
    \Delta \mathcal{L}_{MCapp}  & \coloneqq \mathcal{L}(u_{\mathcal{F}})-\mathcal{L}(u^{*}),         & \Delta \mathcal{L}_{MCtra}  & \coloneqq \lim_{k\to\infty}  \mathbb{E}\left[\mathcal{L}(u^{(k)})-\mathcal{L}(u_{\mathcal{F}})\right].
  \end{align*}
  For DRM-QMC, we define
  \begin{align*}
    \Delta \mathcal{L}_{QMCgen1} & \coloneqq \left\lvert \mathcal{L}(u_{n})-\mathcal{L}_{n}(u_{n})\right\rvert,       & \Delta \mathcal{L}_{QMCgen2} & \coloneqq \left\lvert \mathcal{L}(u_{\mathcal{F}})-\mathcal{L}_{n}(u_{\mathcal{F}}) \right\rvert,   \\
    \Delta \mathcal{L}_{QMCapp}  & \coloneqq \mathcal{L}(u_{\mathcal{F}})-\mathcal{L}(u^{*}), & \Delta \mathcal{L}_{QMCtra}  & \coloneqq \lim_{k\to\infty}  \mathcal{L}(u^{(k)})-\mathcal{L}(u_{\mathcal{F}}).
  \end{align*}
  From the definitions and formulations, we can extract the meanings of three types of errors. The generalization error measures the error incurred by the quadrature method. The approximation error measures how well can $\mathcal{F}$ approximate $H^{1}(\Omega)$. The training error measures the difference between the limit of algorithm output and the optimal approximate of the solution of PDEs in the deep neural network. In the next three subsections, we will analyze these errors separately in detail.

  \subsection{Generalization error}
  There are many papers on the generalization error, where the Rademacher complexity plays an important role. Here we give the definition.
  \begin{definition}
    For a function class $\mathcal{G}$ and a given set $\left\{X_{i}\right\}_{i=1}^{n}$ of independent random samples, we define the empirical Rademacher complexity of $\mathcal{G}$ as
    \begin{equation*}
      \hat{R}_{n}(\mathcal{G})\coloneqq \mathbb{E}_{\varepsilon}\left[\mathop{\sup}_{\mu \in \mathcal{G}}\left\lvert \frac{1}{n}\sum_{j=1}^{n}\varepsilon_{j}\mu(X_{j})\right\rvert \right],
    \end{equation*}
    where $\left\{\varepsilon_{j}\right\}_{j=1}^{n}$ is an independent uniform Bernoulli sequence with $\varepsilon_{j}\in \left\{\pm 1\right\}$. Then the Rademacher complexity of $\mathcal{G}$ is defined as
    \begin{equation*}
      R_{n}(\mathcal{G})\coloneqq \mathbb{E}_{X}\left[\hat{R}_{n}(\mathcal{G})\right].
    \end{equation*}
  \end{definition}

  The Rademacher complexity represents the richness of a function class by measuring the degree to which a hypothesis set can fit random noise on average. High Rademacher complexity indicates that the function class is rich and complex~\cite{Mohri}. Duan et al.~\cite{Duan} and Jian et al.~\cite{Jiao} give an upper bound on the Rademacher complexity of the deep neural network with activation functions different from $\sigma(x)$ in Remark~\ref{remark22}. Based on their works, we present the results under the settings of this paper in Theorem~\ref{thm31}. Before that, we need the following lemmas.
  \begin{lemma}
    For a function class $\mathcal{G}$, we have
    \begin{equation*}
      \mathbb{E} \left[\mathop{\sup}_{\mu \in \mathcal{G}}\left\lvert\frac{1}{n}\sum_{j = 1}^{n}\mu(X_{j})-E[\mu(X)] \right\rvert \right]\leqslant 2R_{n}(\mathcal{G}),
    \end{equation*}
    where $X\thicksim \rm{Unif}$($\Omega$) and $\left\{X_{j}\right\}_{j=1}^{n}$ are i.i.d. uniform random variables on $\Omega$.
  \end{lemma}
  For the rigorous proof of this lemma the reader is referred to~\cite[Proposition 4.11]{Wainwright}. This lemma bounds the worst case error for MC integration by combining the law of large numbers with the Rademacher complexity.
  \begin{lemma} \label{lm32}
    For a given deep neural network $\mathcal{F}$, there exist four positive constants $B_{1}$, $B_{2}$, $L_{1}$, $L_{2}$ such that for any $\theta, \overline{\theta} \in \Theta$ and $x \in \Omega$, the realization of the deep neural network $v(x;\theta)$ has the following properties.
    \begin{enumerate}
      \item[(i)] Boundedness:
      \begin{equation*}
      \left\lvert v(x;\theta)\right\rvert   \leqslant B_{1} \quad and \quad \left\lVert \triangledown_{x} v(x;\theta) \right\rVert_{\ell_{2}}\leqslant B_{2}.
      \end{equation*}
      \item[(ii)] Lipschitz continuity:
      \begin{equation*}
        \left\lvert v(x;\theta)-v(x;\overline{\theta}) \right\rvert     \leqslant L_{1}\left\lVert \theta-\overline{\theta}\right\rVert_{\ell_{2}}
      \end{equation*}
      and
      \begin{equation*}
      \left\lvert  \left\lVert\triangledown_{x} v(x;\theta)\right\rVert_{\ell_{2}}^{2}-\left\lVert\triangledown_{x} v(x;\overline{\theta})\right\rVert_{\ell_{2}}^{2} \right\rvert \leqslant L_{2}\left\lVert \theta-\overline{\theta}\right\rVert_{\ell_{2}}.
      \end{equation*}
    \end{enumerate}
  \end{lemma}
  \begin{proof}
    By routine computation, we obtain that the activation function $\sigma(x)$ and its first-order, second-order derivatives $\sigma'(x),\; \sigma''(x)$ are bounded in bounded closed regions due to their smoothness. Hence, it is easy to see $\sigma(x)$ and $\sigma'(x)$ are Lipschitz continuous in bounded closed regions.
    
    Even though the activation function used in this paper is different from that in ~\cite{Jiao}, we can easily prove that $v(x;\theta)$ and $\triangledown_{x} v(x;\theta)$ are uniformly bounded and Lipschitz continuous with respect to $\theta$ in $\Omega$ by generalizing the results in~\cite[Lemmas 5.9, 5.10, 5.11]{Jiao} under the settings in this paper. Thus the Lipschitz continuity of $\left\lVert\triangledown_{x} v(x;\theta)\right\rVert_{\ell_{2}}^{2}$ follows immediately.
  \end{proof}
  With the uniform boundedness and Lipschitz continuity of $v(x;\theta)$ and $\triangledown_{x} v(x;\theta)$ in $\Omega$, we can prove the following lemma in a similar way to~\cite[Lemma 5.6]{Duan} and~\cite[Theorem 5.13]{Jiao}.
  \begin{lemma}\label{lm33}
    Define the function class
    \begin{equation*}
      \mathcal{F}_{1}\coloneqq\left\{\left\lVert\triangledown_{x} v(\cdot ; \theta)\right\rVert_{\ell_{2}}^{2}:v(\cdot ; \theta) \in \mathcal{F} \right\}.
    \end{equation*}
    We can bound the Rademacher complexities as
    \begin{equation*}
      R_{n}(\mathcal{F}) \leqslant \frac{4}{\sqrt{n}}+\frac{6\sqrt{w}B_{1}}{\sqrt{n}}\sqrt{\log(2L_{1}B_{\Theta}\sqrt{w}\sqrt{n})}
    \end{equation*}
    and
    \begin{equation*}
      R_{n}(\mathcal{F}_{1}) \leqslant \frac{4}{\sqrt{n}}+\frac{6\sqrt{w}B_{2}^{2}}{\sqrt{n}}\sqrt{\log(2L_{2}B_{\Theta}\sqrt{w}\sqrt{n})},
    \end{equation*}
    where $w$ is the total number of nonzero weights.
  \end{lemma}

  \begin{theorem} \label{thm31}
  For a given deep neural network, we have the following results on the the generalization error bounds of DRM.
  \begin{enumerate}
      \item [(i)] For the Poisson equation~(\ref{eq1}), let the function $f$ be bounded in $\Omega$. Then there exists a positive constant $\lambda_{1}$ independent of $n$ such that 
      \begin{equation*}
      \Delta \mathcal{L}_{MCgen1} \leqslant \lambda_{1} n^{-\frac{1}{2}}(\log n)^{\frac{1}{2}} \quad and \quad \Delta \mathcal{L}_{MCgen2} \leqslant \lambda_{1} n^{-\frac{1}{2}}(\log n)^{\frac{1}{2}}.
    \end{equation*}
    \item[(ii)] For the static Schr\"{o}dinger equation~(\ref{eq2}), let the function $g$ be bounded in $\Omega$. Then there exists a positive constant $\lambda_{2}$ independent of $n$ such that
    \begin{equation*}
      \Delta \mathcal{L}_{MCgen1} \leqslant \lambda_{2} n^{-\frac{1}{2}}(\log n)^{\frac{1}{2}} \quad and \quad \Delta \mathcal{L}_{MCgen2} \leqslant \lambda_{2} n^{-\frac{1}{2}}(\log n)^{\frac{1}{2}}.
    \end{equation*}
  \end{enumerate}
  \end{theorem}
  \begin{proof}
    We only prove the upper bound on $\Delta \mathcal{L}_{MCgen1}$ for the Poisson equation~(\ref{eq1}), and the other results can be obtained in a similar way. We have
    \begin{eqnarray*}
      & & \Delta \mathcal{L}_{MCgen1} \\
      & \leqslant & \mathbb{E} \left[\mathop{\sup}_{v \in \mathcal{F}}\left\lvert \mathcal{L}_{P}(v)-\mathcal{L}_{n,P}(v)\right\rvert \right]  \\
      & \leqslant & \mathbb{E}\left[\mathop{\sup}_{v \in \mathcal{F}} \left\lvert \frac{1}{n}\sum_{j = 1}^{n}\frac{1}{2}\left\lVert \triangledown_{x} v(X_{j};\theta)\right\rVert_{\ell_{2}}^{2}- \int\nolimits_{\Omega}\frac{1}{2}\left\lVert \triangledown_{x} v(x;\theta)\right\rVert_{\ell_{2}}^{2}\mathrm{d}x \right\rvert\right]  \\
      & & + \mathbb{E}\left[\mathop{\sup}_{v \in \mathcal{F}}\left\lvert \frac{1}{n}\sum_{j = 1}^{n}f(X_{j})v(X_{j};\theta)-\int\nolimits_{\Omega}f(x)v(x;\theta)\mathrm{d}x\right\rvert \right] \\
      & & + \mathbb{E}\left[\mathop{\sup}_{v \in \mathcal{F}}\left\lvert\frac{1}{2}\left(\frac{1}{n}\sum_{j = 1}^{n}v(X_{j};\theta)\right)^{2}-\frac{1}{2}\left( \int\nolimits_{\Omega}v(x;\theta)\mathrm{d}x\right)^{2} \right\rvert\right] \\
      & \leqslant & R_{n}(\mathcal{F}_{1})+2\sup_{x \in \Omega} \left\lvert f(x)\right\rvert R_{n}(\mathcal{F})+ B_{1}\mathbb{E} \left[\sup_{v\in \mathcal{F}}\left\lvert \frac{1}{n}\sum_{j = 1}^{n}v(X_{j};\theta)-\int\nolimits_{\Omega}v(x;\theta)\mathrm{d}x\right\rvert\right]  \\
      & \leqslant & R_{n}(\mathcal{F}_{1})+2\sup_{x \in \Omega} \left\lvert f(x)\right\rvert R_{n}(\mathcal{F})+2B_{1}R_{n}(\mathcal{F}),
    \end{eqnarray*}
    where $\mathcal{F}_{1}$ is defined as in Lemma~\ref{lm33} and
    \begin{equation*}
      \mathbb{E}\left[\mathop{\sup}_{v \in \mathcal{F}}\left\lvert \frac{1}{n}\sum_{j = 1}^{n}f(X_{j})v(X_{j};\theta)-\int\nolimits_{\Omega}f(x)v(x;\theta)\mathrm{d}x\right\rvert \right]\leqslant 2\sup_{x \in \Omega} \left\lvert f(x)\right\rvert R_{n}(\mathcal{F})
    \end{equation*}
    is a direct result of~\cite[Lemma 5.3]{Duan}. Using Lemmas~\ref{lm32} and~\ref{lm33}, we obtain that there exist a constant $\lambda_{1}$ such that
    \begin{equation*}
      \Delta \mathcal{L}_{MCgen1} \leqslant \lambda_{1} n^{-\frac{1}{2}}(\log n)^{\frac{1}{2}},
    \end{equation*}
    where $\lambda_{1}$ depends on $B_{1}$, $B_{2}$, $L_{1}$, $L_{2}$, $B_{\Theta}$, $w$ and $\sup_{x \in \Omega} \left\lvert f(x)\right\rvert$.
    
    Obviously, $\Delta L_{MCgen2}$ is also bounded by $\mathbb{E} \left[\mathop{\sup}_{v \in \mathcal{F}}\left\lvert \mathcal{L}_{P}(v)-\mathcal{L}_{n,P}(v)\right\rvert \right]$. Then the proof of $(i)$ is completed. The results for the static Schr\"{o}dinger equation~(\ref{eq2}) can be proved in a similar way.
  \end{proof}

  Next, we turn to the generalization error bounds of DRM-QMC. Since every function in $\mathcal{F}$ is smooth, it has bounded Hardy-Krause variation. To investigate the generalization error of DRM-QMC, we prove the following lemma on the uniform boundedness of Hardy-Krause variation for functions $v \in \mathcal{F}$.
  \begin{lemma}\label{lm34}
    There exists a positive constant $C_{1}$ such that for any $v \in \mathcal{F}$,
    \begin{equation*}
      V_{HK}\left(v\right)\leqslant C_{1}.
    \end{equation*}
  \end{lemma}
  \begin{proof}
    By the definition of $\mathcal{F}$, we write
    \begin{equation*}
      v(\cdot;\theta)=T_{L}\circ \sigma \circ T_{L-1} \circ \sigma \circ \cdots \circ \sigma \circ T_{1}(\cdot).
    \end{equation*}
    We consider $d_{1}=d_{2}=\dots=d_{L}=1$ first. By the chain rule, we obtain
    \begin{equation*}
      \frac{\partial v}{\partial x_{i_{1}}}=\frac{d T_{L}(\sigma \circ T_{L-1} \circ \sigma \circ \dots \circ \sigma \circ T_{1})}{d\sigma \circ T_{L-1} \circ \sigma \circ \dots \circ \sigma \circ T_{1}} \frac{d \sigma(T_{L-1} \circ \sigma \circ \dots \circ \sigma \circ T_{1})}{d T_{L-1} \circ \sigma \circ \dots \circ \sigma \circ T_{1}}  \dots \frac{d\sigma(T_{1})}{d T_{1}} \frac{\partial T_{1}}{\partial x_{i_{1}}}.
    \end{equation*}
    Since $\sigma$ is smooth, $\sigma'$ is bounded in the bounded closed regions $T_{i} \circ \sigma \circ \dots \circ \sigma \circ T_{1}(\Omega), \; i=1,2,\dots L-1$. Moreover, $T_{i}'=A_{i},i=2,\dots,L$ and $\partial T_{1}/\partial x_{i_{1}}=A_{1}(1,i_{1})$ are bounded by $B_{\Theta}$. Hence, $\partial v /\partial x_{i_{1}}$ is uniformly bounded for any $v\in\mathcal{F}$.

    Next, we compute partial derivatives of higher order when the width of $\mathcal{F}$ is not necessarily 1. Due to the smoothness of $\sigma(x)$, the $k$-th order derivative of $\sigma(x)$ is bounded in bounded closed regions for any positive integer $k$. Notice that
    \begin{equation*}
      \frac{\partial^{t}T_{i}}{\partial s_{j_{1}}\cdots\partial s_{j_{t}}}=0
    \end{equation*}
    for $t\geqslant 2, 1\leqslant j_{1}<\dots<j_{t}\leqslant d_{i-1}, i=1,2,\dots,L$, and
    \begin{equation*}
    \left\lVert \theta\right\rVert_{\ell_{\infty}}\leqslant B_{\Theta}.
    \end{equation*}
    Hence, $\partial^{k}v(\boldsymbol{x}_{\boldsymbol{i}};\boldsymbol{1}_{\boldsymbol{-i}})/\partial x_{i_{1}}\cdots\partial x_{i_{k}} $ is a sum of a series of derivatives obtained by the chain rule, and its highest order terms are the first order partial derivative of $T_{i},i=1,2,\dots,L$ and $\sigma^{(k)}$, which are all uniformly bounded. Thus $V^{(k)}(v;i_{1},\dots,i_{k}), k=1,2,\dots,d$, are uniformly bounded for any $v \in \mathcal{F}$. The final result follows directly from the definition of the Hardy-Krause variation.
  \end{proof}
  \begin{lemma} \label{coro31}
    \begin{enumerate}
      \item[(i)] If $f \in C^{d}(\Omega)$ in the Poisson equation~(\ref{eq1}), then there exists a positive constant $C_{2}$ such that for any $v \in \mathcal{F}$,
      \begin{equation*}
        V_{HK}\left(\frac{1}{2}\left\lVert \triangledown_{x} v\right\rVert_{\ell_{2}}^{2} -fv\right)\leqslant C_{2}.
        \end{equation*}
      \item[(ii)] If $V \in C^{d}(\Omega)$, $g \in C^{d}(\Omega)$ in the static Schr\"{o}dinger equation~(\ref{eq2}), then there exists a positive constant $C_{3}$ such that for any $v \in \mathcal{F}$,
      \begin{equation*}
        V_{HK}\left(\frac{1}{2}\left\lVert \triangledown_{x} v\right\rVert_{\ell_{2}}^{2} +\frac{1}{2}V\left\lvert v\right\rvert^{2}-gv\right)\leqslant C_{3}.
      \end{equation*}
    \end{enumerate}
  \end{lemma}
  \begin{proof}
    Decompose the Hardy-Krause variation and we have
    \begin{equation*}
    V_{HK}\left(\frac{1}{2}\left\lVert \triangledown_{x} v\right\rVert_{\ell_{2}}^{2} -fv\right)\leqslant \frac{1}{2}V_{HK}\left(\left\lVert \triangledown_{x} v\right\rVert_{\ell_{2}}^{2}\right)+V_{HK}\left(fv\right).
    \end{equation*}
    Since $\sigma(x)$ is smooth, $\triangledown_{x} v \in C^{d}(\Omega)$. We have that $\partial^{k}\left\lVert \triangledown_{x} v\right\rVert_{\ell_{2}}^{2}(\boldsymbol{x}_{\boldsymbol{i}};\boldsymbol{1}_{\boldsymbol{-i}})/\partial x_{i_{1}}\cdots \partial x_{i_{k}}$ is a sum of a series of derivatives obtained by the chain rule. For the same reason as in the proof of Lemma~\ref{lm34}, we obtain that $V^{(k)}(\left\lVert \triangledown_{x} v\right\rVert_{\ell_{2}}^{2};i_{1},\dots,i_{k}), k=1,2,\dots,d$, are uniformly bounded in $\Omega$ for any $v\in \mathcal{F}$. The uniform boundedness of $V_{HK}(\left\lVert \triangledown_{x} v\right\rVert_{\ell_{2}}^{2})$ follows immediately.

    Next, the uniform boundedness of $V_{HK}(fv)$ for any $v \in \mathcal{F}$ is a straightforward consequence of $f \in C^{d}(\Omega)$ and Lemma~\ref{lm34}. Therefore, the proof of~$(i)$ is completed and the proof of~$(ii)$ is similar.
  \end{proof}

  From Lemmas~\ref{lm34} and~\ref{coro31}, we can derive the generalization error bounds of DRM-QMC by the Koksma-Hlawka inequality.
  \begin{theorem}
    For a given deep neural network, we have the following results on the the generalization error bounds of DRM-QMC.
    \begin{enumerate}
        \item [(i)] For the Poisson equation~(\ref{eq1}), let the function $f\in C^{d}(\Omega)$. Then there exists a positive constant $\mu_{1}$ independent of $n$ such that
        \begin{equation*}
        \Delta \mathcal{L}_{QMCgen1}\leqslant \mu_{1} n^{-1}\left(\log n\right)^{d} \quad and \quad \Delta  \mathcal{L}_{QMCgen2}\leqslant \mu_{1} n^{-1}\left(\log n\right)^{d}.
        \end{equation*}
        \item [(ii)]For the static Schr\"{o}dinger equation~(\ref{eq2}), let the function $g \in C^{d}(\Omega)$. Then there exists a positive constant $\mu_{2}$ independent of $n$ such that
        \begin{equation*}
        \Delta \mathcal{L}_{QMCgen1}\leqslant \mu_{2} n^{-1}\left(\log n\right)^{d} \quad and \quad \Delta \mathcal{L}_{QMCgen2}\leqslant \mu_{2} n^{-1}\left(\log n\right)^{d}.
        \end{equation*}
    \end{enumerate}
  \end{theorem}
  \begin{proof}
    For the Poisson equation~(\ref{eq1}), we have
    \begin{small}
    \begin{eqnarray*}
      & & \Delta \mathcal{L}_{QMCgen1} \\
      & \leqslant & \sup_{v \in \mathcal{F}}\left\lvert \mathcal{L}_{P}(v)-\mathcal{L}_{n,P}(v)\right\rvert \\
      & \leqslant & \sup_{v \in \mathcal{F}}\left\lvert \int\nolimits_{\Omega} \left(\frac{\left\lVert \triangledown_{x} v(x;\theta)\right\rVert_{\ell_{2}}^{2}}{2} -f(x)v(x;\theta)\right)\mathrm{d}x - \frac{1}{n}\sum_{j = 1}^{n}\left(\frac{\left\lVert \triangledown_{x} v(X_{j};\theta)\right\rVert_{\ell_{2}}^{2}}{2} - f(X_{j})v(X_{j};\theta)\right)\right\rvert \\
      & & +\sup_{v \in \mathcal{F}}\left\lVert v\right\rVert_{L^{\infty}(\Omega)} \sup_{v\in \mathcal{F}}\left\lvert \int\nolimits_{\Omega}v(x;\theta)\mathrm{d}x - \frac{1}{n}\sum_{j = 1}^{n}v(X_{j};\theta)\right\rvert,
    \end{eqnarray*}
    \end{small}
    where $X_{1},\dots,X_{n}$ are low discrepancy points. Using the Koksma-Hlawka inequality, which is ensured by Lemmas~\ref{lm34} and~\ref{coro31}, we obtain an upper bound on $\Delta \mathcal{L}_{QMCgen1}$, i.e.,
    \begin{equation*}
      \Delta \mathcal{L}_{QMCgen1}\leqslant \left(C_{2}+B_{1}C_{1}\right) D_{n}^{*}.
    \end{equation*}
    Hence, there exists a positive constant $\mu_{1}$ such that
    \begin{equation*}
      \Delta \mathcal{L}_{QMCgen1}\leqslant \mu_{1} n^{-1}\left(\log n\right)^{d},
    \end{equation*}
    where $\mu_{1}$ depends on $B_{1}$, $C_{1}$, $C_{2}$ and the construction of the used low discrepancy sequence. 
    
    For $\Delta \mathcal{L}_{QMCgen2}$, we notice that it is also bounded by $\sup_{v \in \mathcal{F}}\left\lvert \mathcal{L}_{P}(v)-L_{n,P}(v)\right\rvert$, so $\Delta \mathcal{L}_{QMCgen1}$ and $\Delta \mathcal{L}_{QMCgen2}$ have the same upper bound.
    
    The results for the case of the static Schr\"{o}dinger equation~(\ref{eq2}) can be proved in a similar way.
  \end{proof}
  
  \subsection{Approximation error}
  In this subsection, we present the approximation error bound. From Theorem~\ref{thm21}, we can derive the following error bounds naturally.

  For the Poisson equation~(\ref{eq1}), we have
  \begin{equation*}
    \mathcal{L}(u_{\mathcal{F}})-\mathcal{L}(u_{P}^{*})\leqslant \frac{1}{2} \inf_{v \in \mathcal{F}}\left\lVert v-u_{P}^{*}\right\rVert_{H^{1}(\Omega)}^{2}.
  \end{equation*}

  For the static Schr\"{o}dinger equation~(\ref{eq2}), we have
  \begin{equation*}
    \mathcal{L}(u_{\mathcal{F}})-\mathcal{L}(u_{S}^{*})\leqslant \frac{\max\left\{1,V_{max}\right\}}{2} \inf_{v \in \mathcal{F}}\left\lVert v-u_{S}^{*}\right\rVert_{H^{1}(\Omega)}^{2}.
  \end{equation*}
  To give the approximation error bound, we need to find an upper bound on
  \begin{equation}
    \nonumber
    \inf_{v \in \mathcal{F}}\left\lVert v-u\right\rVert_{H^{1}(\Omega)}^{2}
  \end{equation}
  for a given function $u \in H^{1}\left(\Omega\right)$. With reference to~\cite[Proposition 4.8]{Guhring}, we have the following theorem.
  \begin{theorem} \label{thm33}
    For $s\geqslant 2$, there exist constants $L,C,\vartheta ,\widetilde{\varepsilon}$ depending on $d,s,$ such that for any $u \in H^{s}(\Omega)$ with $\left\lVert u\right\rVert_{H^{s}\left(\Omega\right)} \leqslant 1$ and every $\varepsilon \in (0,\widetilde{\varepsilon})$, there is a realization $v$ in the deep neural network $\mathcal{F}$ with depth at most $L$ and at most $C\varepsilon^{-d/(s-1)}$ nonzero weights such that
    \begin{equation*}
      \left\lVert u-v\right\rVert_{H^{1}(\Omega)} \leqslant \varepsilon.
    \end{equation*}
    Moreover, the absolute values of weights are bounded by $C\varepsilon^{-\vartheta}$.
  \end{theorem}
  Hence, it is reasonable to consider the deep neural network with bounded weights, which is consistent with the construction of the deep neural network in this paper. When we extend it to a more general case, i.e., to estimate a function $u \in H^{1}(\Omega)$, there is no satisfactory conclusion yet as far as we know. Because the approximation error is independent of the sampling strategy, this part is not our concern.
  
  \subsection{Training error}
  To compare the training errors with respect to different sampling strategies, we study the convergence order on the mini-batch size $n$ and the dimension $d$. We establish the convergence rates in the case of the Poisson equation~(\ref{eq1}) first and the results of the static Schr\"{o}dinger equation~(\ref{eq2}) follow immediately.

  For the Poisson equation (\ref{eq1}), we recall the loss functional and the iteration of weight parameters
  \begin{equation*}
    \mathcal{L}_{P}(\theta)\coloneqq\int\nolimits_{\Omega}\left(\frac{1}{2}\left\lVert \triangledown_{x} v(x;\theta)\right\rVert_{\ell_{2}}^{2} -f(x)v(x;\theta)\right)\mathrm{d}x+\frac{1}{2}\left( \int\nolimits_{\Omega}v(x;\theta)\mathrm{d}x\right)^{2},
  \end{equation*}
  \begin{equation*}
    \theta_{k+1}=\theta_{k}-\alpha_{k} G_{P}(\theta_{k},\xi_{k}),
  \end{equation*}
  where $G_{P}(\theta_{k},\xi_{k})$ is defined as
  \begin{equation}\label{eqnew}
    \triangledown_{\theta} \left\{\frac{1}{n}\sum_{j = 1}^{n}\left(\frac{1}{2}\left\lVert \triangledown_{x} v(X_{j,k};\theta_{k})\right\rVert_{\ell_{2}}^{2} - f(X_{j,k})v(X_{j,k};\theta_{k})\right) + \frac{1}{2}\left(\frac{1}{n}\sum_{j = 1}^{n}v(X_{j,k};\theta_{k})\right)^{2}\right\}
  \end{equation}
  and $n=2^{\tau}, \tau \in \mathbb{N}^{+}$ and
  $\xi_{k}$ represents $\left\{X_{j,k}\right\}_{j=1}^{n}$, for $k=0,1,2,\dots$. Since $v(\cdot;\theta)$ and $\theta$ are one-to-one, we abbreviate $\mathcal{L}_{P}(v(\cdot;\theta))$ to $\mathcal{L}_{P}(\theta)$. To prove the Lipschitz continuity of $\triangledown_{\theta}\mathcal{L}_{P}(\theta)$, we need the following lemma.
  \begin{lemma} \label{lm35}
    There exist positive constants $L_{3},L_{4},L_{5},L_{6}$ such that $\forall x \in \Omega$, $\forall \theta, \overline{\theta} \in \Theta$,
    \begin{equation*}
      \left\lVert \triangledown_{\theta}v(x;\theta)-\triangledown_{\theta}v(x;\overline{\theta})\right\rVert_{\ell_{2}} \leqslant L_{3} \left\lVert \theta - \overline{\theta} \right\rVert_{\ell_{2}},
    \end{equation*}
    \begin{equation*}
      \left\lVert \triangledown_{x}v(x;\theta)-\triangledown_{x}v(x;\overline{\theta})\right\rVert_{\ell_{2}} \leqslant L_{4} \left\lVert \theta - \overline{\theta} \right\rVert_{\ell_{2}},
    \end{equation*}
    \begin{equation*}
      \left\lVert \triangledown_{\theta}\triangledown_{x}v(x;\theta)-\triangledown_{\theta}\triangledown_{x}v(x;\overline{\theta})\right\rVert_{\ell_{2}}  \leqslant  L_{5} \left\lVert \theta - \overline{\theta} \right\rVert_{\ell_{2}},
    \end{equation*}
    and
    \begin{equation*}
      \left\lVert \triangledown_{\theta}\left[\frac{1}{2}\left(\int_{\Omega}v(x;\theta)\mathrm{d}x\right)^{2}\right] -\triangledown_{\theta}\left[\frac{1}{2}\left(\int_{\Omega}v(x;\overline{\theta})\mathrm{d}x\right)^{2}\right]\right\rVert_{\ell_{2}}  \leqslant L_{6} \left\lVert \theta - \overline{\theta} \right\rVert_{\ell_{2}}.
    \end{equation*}
  \end{lemma}
  \begin{proof}
    By the construction of $\mathcal{F}$, we have that $v(x;\theta)$ and $\triangledown_{x}v(x;\theta)$ are smooth with respect to $\theta$ in the bounded closed region $\Theta$. Hence, the first three inequalities are true. To prove the fourth inequality, exchanging the order of integration and differentiation, we have
    \begin{eqnarray*}
      & & \left\lVert \triangledown_{\theta}\left[\frac{1}{2}\left(\int_{\Omega}v(x;\theta)\mathrm{d}x\right)^{2}\right] -\triangledown_{\theta}\left[\frac{1}{2}\left(\int_{\Omega}v(x;\overline{\theta})\mathrm{d}x\right)^{2}\right]\right\rVert_{\ell_{2}} \\
      & =&\left\lVert \int_{\Omega}v(x;\theta)\mathrm{d}x \int_{\Omega}\triangledown_{\theta}v(x;\theta)\mathrm{d}x- \int_{\Omega}v(x;\overline{\theta})\mathrm{d}x \int_{\Omega}\triangledown_{\theta}v(x;\overline{\theta})\mathrm{d}x \right\rVert_{\ell_{2}}  \\
      & \leqslant& \left\lvert \int_{\Omega}v(x;\theta)\mathrm{d}x \right\rvert  \left\lVert \int_{\Omega}\left(\triangledown_{\theta}v(x;\theta)-\triangledown_{\theta}v(x;\overline{\theta})\right)\mathrm{d}x \right\rVert_{\ell_{2}} \\
      & & +\left\lVert \int_{\Omega} \triangledown_{\theta}v(x;\overline{\theta})\mathrm{d}x \right\rVert_{\ell_{2}}\left\lvert \int_{\Omega}\left(v(x;\theta)-v(x;\overline{\theta})\right)\mathrm{d}x \right\rvert \\
      & \leqslant& B_{1}L_{3} \left\lVert \theta - \overline{\theta}  \right\rVert_{\ell_{2}}+ L_{1}\sup_{v \in \mathcal{F}} \left\lVert \int_{\Omega} \triangledown_{\theta}v(x;\overline{\theta})\mathrm{d}x \right\rVert_{\ell_{2}} \left\lVert \theta - \overline{\theta}  \right\rVert_{\ell_{2}}. \\
    \end{eqnarray*}
    Since $v(x;\theta)$ is smooth with respect to $\theta$ and $\Theta$ is bounded and closed, we have that there exists a positive constant $B_{3}$ such that $\forall x \in \Omega$, $\forall \theta \in \Theta$,
    \begin{equation*}
      \left\lVert\triangledown_{\theta}v(x;\theta) \right\rVert_{\ell_{2}}\leqslant B_{3}.
    \end{equation*}
    Hence, the proof is completed.
  \end{proof}
  \begin{theorem} \label{thm34}
    If f is bounded in $\Omega$, then the gradient function of $\mathcal{L}_{P}(\theta)$ is Lipschitz continuous with Lipschitz constant $L_{0}>0$, i.e., $\forall \theta, \overline{\theta} \in \Theta$,
    \begin{equation*}
      \left\lVert \triangledown_{\theta}\mathcal{L}_{P}(\theta)-\triangledown_{\theta}\mathcal{L}_{P}(\overline{\theta})\right\rVert_{\ell_{2}}\leqslant L_{0}\left\lVert \theta - \overline{\theta} \right\rVert_{\ell_{2}}.
    \end{equation*}
  \end{theorem}
  \begin{proof}
    Firstly, we exchange the order of integration and differentiation,
    \begin{equation}
      \nonumber
      \triangledown_{\theta}\mathcal{L}_{P}(\theta)= \int_{\Omega}\left(\triangledown_{\theta} \triangledown_{x}v(x;\theta) \triangledown_{x} v(x;\theta)-f(x) \triangledown_{\theta}v(x;\theta)\right)\mathrm{d}x+ \triangledown_{\theta}\left[\frac{1}{2}\left(\int_{\Omega}v(x;\theta)\mathrm{d}x\right)^{2}\right].
    \end{equation}
    Using the smoothness of $v(x;\theta)$ with respect to $x$ and $\theta$, we obtain that there exist a positive constant $B_{4}$ such that $\forall x \in \Omega$, $\forall \theta \in \Theta$,
    \begin{equation*}
      \left\lVert \triangledown_{\theta} \triangledown_{x} v(x;\theta)\right\rVert_{\ell_{2}}\leqslant B_{4}.
    \end{equation*}
    Hence, it follows from the triangle inequality and Lemma \ref{lm35} that
    \begin{eqnarray*}
      & & \left\lVert\triangledown_{\theta}\mathcal{L}_{P}(\theta)-\triangledown_{\theta}\mathcal{L}_{P}(\overline{\theta})\right\rVert_{\ell_{2}} \\
      & \leqslant & \left\lVert \int_{\Omega}\left(\triangledown_{\theta}\triangledown_{x}v(x;\theta)  \triangledown_{x} v(x;\theta)-\triangledown_{\theta} \triangledown_{x}v(x;\overline{\theta})  \triangledown_{x} v(x;\theta)\right)\mathrm{d}x\right\rVert_{\ell_{2}}  \\
      & & +\left\lVert \int_{\Omega}\left(\triangledown_{\theta}\triangledown_{x}v(x;\overline{\theta}) \triangledown_{x} v(x;\theta)-\triangledown_{\theta}\triangledown_{x}v(x;\overline{\theta}) \triangledown_{x} v(x;\overline{\theta})\right)\mathrm{d}x\right\rVert_{\ell_{2}} \\
      & & +\left\lVert \int_{\Omega}\left(f(x) \left(\triangledown_{\theta}v(x;\theta)-\triangledown_{\theta}v(x;\overline{\theta})\right) \right)\mathrm{d}x \right\rVert_{\ell_{2}}  \\
      & & +\left\lVert \triangledown_{\theta}\left[\frac{1}{2}\left(\int_{\Omega}v(x;\theta)\mathrm{d}x\right)^{2}\right] -\triangledown_{\theta}\left[\frac{1}{2}\left(\int_{\Omega}v(x;\overline{\theta})\mathrm{d}x\right)^{2}\right]\right\rVert_{\ell_{2}} \\
      & \leqslant & B_{2}L_{5} \left\lVert \theta - \overline{\theta} \right\rVert_{\ell_{2}}+B_{4}L_{4} \left\lVert \theta - \overline{\theta} \right\rVert_{\ell_{2}} + L_{3} \max_{x\in \Omega}\left\lvert f(x)\right\rvert  \left\lVert \theta - \overline{\theta} \right\rVert_{\ell_{2}} +L_{6} \left\lVert \theta - \overline{\theta} \right\rVert_{\ell_{2}}
    \end{eqnarray*}
    for any $\theta, \overline{\theta}$ in $\Theta$. Taking $L_{0}=B_{2}L_{5}+B_{4}L_{4}+L_{3} \max_{x\in \Omega}\left\lvert f(x)\right\rvert+L_{6}$ yields the result.
  \end{proof}
  Using the method in~\cite[Appendix B.]{Bottou}, we have the following consequence of Theorem~\ref{thm34}. The loss functional value $\mathcal{L}_{P}(\theta)$ can be bounded as
  \begin{equation} \label{eq5}
    \mathcal{L}_{P}(\theta) \leqslant \mathcal{L}_{P}(\overline{\theta})+\triangledown_{\theta}\mathcal{L}_{P}(\overline{\theta})^{T}(\theta-\overline{\theta})+\frac{L_{0}}{2}\left\lVert \theta - \overline{\theta} \right\rVert_{\ell_{2}}^{2}.
  \end{equation}
  
  \begin{lemma} \label{lm36}
    Assume that there exists a constant $0<c<L_{0}$ such that $\forall \theta, \overline{\theta} \in \Theta$,
    \begin{equation} \label{eqnew33}
      \mathcal{L}_{P}(\overline{\theta})\geqslant \mathcal{L}_{P}(\theta)+\triangledown_{\theta}\mathcal{L}_{P}(\theta)^{T}(\overline{\theta}-\theta)+\frac{c}{2}\left\lVert \overline{\theta}-\theta \right\rVert_{\ell_{2}}^{2}.
    \end{equation}
    Then $\mathcal{L}_{P}(\theta)$ has a unique minimizer, denoted by $\theta^{*}$. Moreover, $\forall \theta \in \Theta$, we have
    \begin{equation*}
      \mathcal{L}_{P}(\theta)-\mathcal{L}_{P}(\theta^{*}) \leqslant \frac{1}{2c}\left\lVert \triangledown_{\theta}\mathcal{L}_{P}(\theta)\right\rVert_{\ell_{2}}^{2}.
    \end{equation*}
  \end{lemma}
  Typically, (\ref{eqnew33}) is called the strong convexity of $\mathcal{L}_{P}(\theta)$, and the proof can be found in~\cite[Theorem 5.25]{Beck} and~\cite[Appendix B.]{Bottou}. If $c= L_{0}$, the results are a little different, and we give the results in that case in Remark~\ref{remark33}.
  
  We now establish the training error of DRM when the stepsize is fixed. Before that, we need the following lemma on the variance of the gradient estimator $G_{P}(\theta_{k},\xi_{k})$ defined as~(\ref{eqnew}).

  \begin{lemma} \label{lm37}
     For solving the Poisson equation~(\ref{eq1}), we have
  \begin{equation*}
    \mathbb{E}\left[\mathrm{trace}\left(\mathrm{Var}\left[G_{P}(\theta_{k},\xi_{k})| \theta_{k}\right]\right)\right]=O\left(n^{-1}\right),
  \end{equation*}
  where the expectation is taken with respect to $\left\{\xi_{k}\right\}_{k=0}^{\infty}$.
  \end{lemma}
  \begin{proof}
  For $G_{P}(\theta_{k},\xi_{k})$, we can only consider its dominant term when we study the order of $\mathbb{E}\left[\mathrm{trace}\left(\mathrm{Var}\left[G_{P}(\theta_{k},\xi_{k})| \theta_{k}\right]\right)\right]$ with respect to $n$, namely the first term of $G_{P}(\theta_{k},\xi_{k})$
  \begin{equation*}
      \frac{1}{n}\sum_{j = 1}^{n}\left(\frac{1}{2}\left\lVert \triangledown_{x} v(X_{j,k};\theta_{k})\right\rVert_{\ell_{2}}^{2} - f(X_{j,k})v(X_{j,k};\theta_{k})\right)\eqqcolon\frac{1}{n}\sum_{j = 1}^{n}\mu_{P}(X_{j,k};\theta_{k}).
  \end{equation*}
    By the central limit theorem and some routine computation, we have
    \begin{equation*}
        \mathrm{trace}\left(\mathrm{Var}\left[G_{P}(\theta_{k},\xi_{k})| \theta_{k}\right]\right)=\frac{h(\theta_{k})}{n},
    \end{equation*}
    where 
    \begin{equation*}
      h(\theta_{k})=\mathrm{Var}\left[\frac{1}{2}\left\lVert \triangledown_{x} v(X;\theta_{k})\right\rVert_{\ell_{2}}^{2} - f(X)v(X;\theta_{k})|\theta_{k}\right].
    \end{equation*}
    
    Using the fact that $\theta_{0}$ is a constant, for $n=0$,
    \begin{equation*}
       \mathbb{E}\left[\mathrm{trace}\left(\mathrm{Var}\left[G_{P}(\theta_{0},\xi_{0})| \theta_{0}\right]\right)\right]=\mathbb{E}\left[\frac{h(\theta_{0})}{n}\right]=O\left(n^{-1}\right).
    \end{equation*}
    For $n=1$,
    \begin{eqnarray*}
     \mathbb{E}\left[\mathrm{trace}\left(\mathrm{Var}\left[G_{P}(\theta_{1},\xi_{1})| \theta_{1}\right]\right)\right]&=&\mathbb{E}\left[\frac{h(\theta_{0}-\alpha_{0} n^{-1}\sum_{j = 1}^{n}\mu_{P}(X_{j,0};\theta_{0}))}{n}\right]\\
     &=&\mathbb{E}\left[\frac{h(\theta_{0})-\alpha_{0} n^{-1}\sum_{j = 1}^{n}\mu_{P}(X_{j,0};\theta_{0})h'(\theta_{0})+\cdots}{n}\right]\\
     &=&O\left(n^{-1}\right),
    \end{eqnarray*}
    where the second equality follows from the Taylor expansion. By the induction method, we complete the proof.
  \end{proof}

  \begin{theorem}\label{thm35}
    Under the same assumptions given in Lemma~\ref{lm36}, for solving the Poisson equation~(\ref{eq1}), if the iteration stepsize $\alpha$ is a constant satisfying $0<\alpha < 2/L_{0}$, then the training error of DRM satisfies
    \begin{equation*}
      \Delta \mathcal{L}_{MCtra} = O\left(n^{-1}\right).
    \end{equation*}
  \end{theorem}
  \begin{proof}
    For MC, using the fact that $\theta_{i}$ is independent of $\xi_{k}$ for $i\leqslant k$, we have that
    \begin{equation*}
      \mathbb{E} \left[G_{P}(\theta_{k},\xi_{k})| \theta_{k}\right] =\triangledown_{\theta}\mathcal{L}_{P}(\theta_{k}).
    \end{equation*}
    By the definition of conditional variance, it holds
    \begin{eqnarray*}
      \mathbb{E}\left[\left\lVert G_{P}(\theta_{k},\xi_{k}) \right\rVert_{\ell_{2}}^{2} | \theta_{k}\right] & = & \mathrm{trace}\left(\mathrm{Var}\left[G_{P}(\theta_{k},\xi_{k})| \theta_{k}\right]\right) + \left\lVert \mathbb{E}\left[G_{P}(\theta_{k},\xi_{k})| \theta_{k}\right]  \right\rVert_{\ell_{2}}^{2}\\
      & = &O\left(n^{-1}\right)+\left\lVert \triangledown_{\theta}\mathcal{L}_{P}(\theta_{k})\right\rVert_{\ell_{2}}^{2}. \\
    \end{eqnarray*}
    Based on~(\ref{eq5}), we can derive the difference between  the loss functional values in two adjacent steps with fixed stepsize as
    \begin{eqnarray*}
      \mathcal{L}_{P}(\theta_{k+1})-\mathcal{L}_{P}(\theta_{k}) & \leqslant& \triangledown_{\theta}\mathcal{L}_{P}(\theta_{k})^{T}  (\theta_{k+1}-\theta_{k})+\frac{L_{0}}{2}\left\lVert \theta_{k+1}-\theta_{k} \right\rVert_{\ell_{2}}^{2} \\
      & = &-\alpha \triangledown_{\theta}\mathcal{L}_{P}(\theta_{k})^{T} G_{P}(\theta_{k},\xi_{k}) +\frac{\alpha^{2}L_{0}}{2}\left\lVert G_{P}(\theta_{k},\xi_{k})\right\rVert_{\ell_{2}}^{2}.
    \end{eqnarray*}
    Taking conditional expectations for a given $\theta_{k}$ on both sides, we have
    \begin{eqnarray*}
      & & \mathbb{E}\left[\mathcal{L}_{P}(\theta_{k+1})| \theta_{k}\right]  -\mathcal{L}_{P}(\theta_{k}) \\
      & \leqslant &-\alpha \triangledown_{\theta}\mathcal{L}_{P}(\theta_{k})^{T} \mathbb{E}\left[G_{P}(\theta_{k},\xi_{k})| \theta_{k}\right] +\frac{\alpha^{2}L_{0}}{2} \mathbb{E}\left[\left\lVert G_{P}(\theta_{k},\xi_{k})\right\rVert_{\ell_{2}}^{2}| \theta_{k}\right]    \\
      & = &-\alpha \left\lVert \triangledown_{\theta}\mathcal{L}_{P}(\theta_{k})\right\rVert_{\ell_{2}}^{2}+\frac{\alpha^{2}L_{0}}{2}\left[\mathrm{trace}\left(\mathrm{Var}\left[G_{P}(\theta_{k},\xi_{k})| \theta_{k}\right]\right)+\left\lVert \triangledown_{\theta}\mathcal{L}_{P}(\theta_{k})\right\rVert_{\ell_{2}}^{2}\right] \\
      & = &\left(\frac{\alpha^{2}L_{0}}{2}-\alpha \right) \left\lVert \triangledown_{\theta}\mathcal{L}_{P}(\theta_{k})\right\rVert_{\ell_{2}}^{2}+\frac{\alpha^{2}L_{0}}{2} \mathrm{trace}\left(\mathrm{Var}\left[G_{P}(\theta_{k},\xi_{k})| \theta_{k}\right]\right) \\
      & \leqslant & \left(\alpha^{2}L_{0}c-2\alpha c\right) \left(\mathcal{L}_{P}(\theta_{k})-\mathcal{L}_{P}(\theta^{*})\right)+\frac{\alpha^{2}L_{0}}{2} \mathrm{trace}\left(\mathrm{Var}\left[G_{P}(\theta_{k},\xi_{k})| \theta_{k}\right]\right).
    \end{eqnarray*}
    The last inequality follows from $\alpha^{2}L_{0}/2-\alpha < 0$ and Lemma~\ref{lm36}. A routine computation gives rise to the following inequality
    \begin{eqnarray*}
      & \mspace{21mu} & \mathbb{E}\left[\mathcal{L}_{P}(\theta_{k+1})| \theta_{k}\right]-\mathcal{L}_{P}(\theta^{*})+\frac{\alpha L_{0}\mathrm{trace}\left(\mathrm{Var}\left[G_{P}(\theta_{k},\xi_{k})| \theta_{k}\right]\right)}{2\alpha L_{0}c-4c} \\
      & \leqslant & \left(\alpha^{2}L_{0} c-2 \alpha c+1\right) \left[\mathcal{L}_{P}(\theta_{k})-\mathcal{L}_{P}(\theta^{*})+\frac{\alpha L_{0}\mathrm{trace}\left(\mathrm{Var}\left[G_{P}(\theta_{k},\xi_{k})| \theta_{k}\right]\right)}{2\alpha L_{0}c-4c}\right].
    \end{eqnarray*}
    From $0<\alpha<2/L_{0}$ and $0<c< L_{0}$, we can derive that $0<\alpha^{2}L_{0} c-2 \alpha c+1<1$. Using Lemma~\ref{lm37} and taking mathematical expectations with respect to $\left\{\xi_{k}\right\}_{k=0}^{\infty}$, we obtain
    \begin{equation*}
      \mathbb{E}\left[\mathrm{trace}\left(\mathrm{Var}\left[G_{P}(\theta_{k},\xi_{k})| \theta_{k}\right]\right)\right]=O\left(n^{-1}\right).
    \end{equation*}
    Hence, the expected optimality gap satisfies
    \begin{equation*}
      \lim_{k \to +\infty} \mathbb{E}[\mathcal{L}_{P}(\theta_{k})-\mathcal{L}_{P}(\theta^{*})]=O\left(n^{-1}\right).
    \end{equation*}
  \end{proof}
  
  To analyze the training error of DRM-QMC, we need the following lemma on the gradient $\triangledown_{\theta}\mathcal{L}_{P}(\theta_{k})$ and its estimator $G_{P}(\theta_{k},\xi_{k})$.
  \begin{lemma} \label{lm38}
    Let $\left\{\xi_{k}\right\}_{k=0}^{\infty}$ be a low discrepancy sequence and $f\in C^{d}(\Omega)$. There exists a positive constant $C_{4}$ such that $\forall k \in \mathbb{N}$,
    \begin{enumerate}
      \item[(i)] \begin{equation*}
              \left\lVert G_{P}(\theta_{k},\xi_{k})-\triangledown_{\theta}\mathcal{L}_{P}(\theta_{k}) \right\rVert_{\ell_{2}} \leqslant C_{4} n^{-1}\left(\log n\right)^{d},
            \end{equation*}
      \item[(ii)] \begin{equation*}
              \triangledown_{\theta}\mathcal{L}_{P}(\theta_{k})^{T} G_{P}(\theta_{k},\xi_{k}) \geqslant
              \left\lVert \triangledown_{\theta}\mathcal{L}_{P}(\theta_{k})\right\rVert_{\ell_{2}}^{2}-C_{4} n^{-1}\left(\log n\right)^{d} \left\lVert \triangledown_{\theta}\mathcal{L}_{P}(\theta_{k})\right\rVert_{\ell_{2}}.
            \end{equation*}
    \end{enumerate}
  \end{lemma}
  \begin{proof}
    For~$(i)$, proceeding as in the proof of Lemmas~\ref{lm34} and~\ref{coro31}, we obtain that there exist positive constants $C_{5}$ and $C_{6}$ such that $\forall \theta \in \Theta$,
    \begin{equation*}
      V_{HK}\left(\triangledown_{\theta}\triangledown_{x}v(x;\theta) \triangledown_{x} v(x;\theta)-f(x) \triangledown_{\theta}v(x;\theta)\right) \leqslant C_{5}
    \end{equation*}
    and
    \begin{equation*}
      V_{HK}\left(\triangledown_{\theta}v(x;\theta)  \right) \leqslant C_{6}.
    \end{equation*}
    Using the Koksma-Hlawka inequality and Lemmas~\ref{lm32} and~\ref{lm34}, we have
    \begin{eqnarray*}
      &\mspace{21mu}&\left\lVert G_{P}(\theta_{k},\xi_{k})-\triangledown_{\theta}\mathcal{L}_{P}(\theta_{k}) \right\rVert_{\ell_{2}}\\
      &\leqslant & \frac{1}{2} \left\lVert \triangledown_{\theta}\left(\frac{1}{n}\sum_{j = 1}^{n}v(X_{j,k};\theta_{k})\right)^{2}-\triangledown_{\theta} \left( \int\nolimits_{\Omega}v(x;\theta)\mathrm{d}x\right)^{2}\right\rVert_{\ell_{2}}+ C_{5}D_{n}^{*}\\
      &\leqslant & \left\lVert \frac{1}{n}\sum_{j = 1}^{n}v(X_{j,k};\theta_{k}) \frac{1}{n}\sum_{j = 1}^{n} \triangledown_{\theta} v(X_{j,k};\theta_{k})-\frac{1}{n}\sum_{j = 1}^{n}v(X_{j,k};\theta_{k}) \int\nolimits_{\Omega}\triangledown_{\theta} v(x;\theta)\mathrm{d}x\right\rVert_{\ell_{2}} \\
      &\mspace{21mu}& +\left\lVert \frac{1}{n}\sum_{j = 1}^{n}v(X_{j,k};\theta_{k}) \int\nolimits_{\Omega}\triangledown_{\theta} v(x;\theta)\mathrm{d}x- \int\nolimits_{\Omega}v(x;\theta)\mathrm{d}x\int\nolimits_{\Omega}\triangledown_{\theta} v(x;\theta)\mathrm{d}x\right\rVert_{\ell_{2}}+ C_{5}D_{n}^{*}\\
      &\leqslant & B_{1}\left\lVert \frac{1}{n}\sum_{j = 1}^{n} \triangledown_{\theta} v(X_{j,k};\theta_{k})-\int\nolimits_{\Omega}\triangledown_{\theta} v(x;\theta)\mathrm{d}x\right\rVert_{\ell_{2}} \\
      & & + B_{3}\left\lvert \frac{1}{n}\sum_{j = 1}^{n}v(X_{j,k};\theta_{k})-\int\nolimits_{\Omega}v(x;\theta)\mathrm{d}x \right\rvert + C_{5}D_{n}^{*} \\
      &\leqslant & (B_{1}C_{6}+B_{3}C_{1}+C_{5})D_{n}^{*}.
    \end{eqnarray*}
    Notice that $D_{n}^{*}=O\left(n^{-1}\left(\log n\right)^{d}\right)$, so~$(i)$ is proved.

    For~$(ii)$, using the result of~$(i)$, we obtain
    \begin{equation}\label{eq6}
      \triangledown_{\theta}\mathcal{L}_{P}(\theta_{k})^{T} G_{P}(\theta_{k},\xi_{k}) \geqslant \frac{1}{2} \left(\left\lVert \triangledown_{\theta}\mathcal{L}_{P}(\theta_{k})\right\rVert_{\ell_{2}}^{2} + \left\lVert G_{P}(\theta_{k},\xi_{k})\right\rVert_{\ell_{2}}^{2}-C_{4}^{2}n^{-2} \left(\log n\right)^{2d}\right)
    \end{equation}
    and
    \begin{equation}\label{eq7}
      \left\lVert G_{P}(\theta_{k},\xi_{k})\right\rVert_{\ell_{2}}^{2} \geqslant \left\lVert \triangledown_{\theta}\mathcal{L}_{P}(\theta_{k})\right\rVert_{\ell_{2}}^{2}+C_{4}^{2} n^{-2}\left(\log n\right)^{2d}-2C_{4} n^{-1}\left(\log n\right)^{d} \left\lVert \triangledown_{\theta}\mathcal{L}_{P}(\theta_{k})\right\rVert_{\ell_{2}}.
    \end{equation}
    By substituting~(\ref{eq7}) into~(\ref{eq6}), we complete the proof.
  \end{proof}
  \begin{theorem} \label{thm36}
    Under the same assumption given in Lemma~\ref{lm36}, for solving the Poisson equation~(\ref{eq1}), if the iteration stepsize $\alpha$ is a constant satisfying $0<\alpha < 2/L_{0}$ and $f\in C^{d}(\Omega)$, then the training error of DRM-QMC satisfies
    \begin{equation*}
      \Delta \mathcal{L}_{QMCtra}=O\left(n^{-1}(\log n)^{d} \right).
    \end{equation*}
  \end{theorem}
  \begin{proof}
    Using Lemma~\ref{lm38}~$(i)$, we have
    \begin{equation*}
      \left\lVert G_{P}(\theta_{k},\xi_{k})\right\rVert_{\ell_{2}}^{2} \leqslant \left\lVert \triangledown_{\theta}\mathcal{L}_{P}(\theta_{k})\right\rVert_{\ell_{2}}^{2}+C_{4}^{2} n^{-2}\left(\log n\right)^{2d}+2C_{4} n^{-1}\left(\log n\right)^{d} \left\lVert \triangledown_{\theta}\mathcal{L}_{P}(\theta_{k})\right\rVert_{\ell_{2}}.
    \end{equation*}
    Based on the proof of Theorem~\ref{thm34} and $f\in C^{d}(\Omega)$, we have that $\left\lVert \triangledown_{\theta}\mathcal{L}_{P}(\theta)\right\rVert_{\ell_{2}}$ is bounded in $\Theta$. Hence, let $M_{P}\coloneqq \sup_{\theta \in \Theta} \left\lVert \triangledown_{\theta}\mathcal{L}_{P}(\theta)\right\rVert_{\ell_{2}}$ and there holds
    \begin{eqnarray*}
      &\mspace{21mu}&\mathcal{L}_{P}(\theta_{k+1})-\mathcal{L}_{P}(\theta_{k})\\
      &\leqslant& -\alpha \triangledown_{\theta}\mathcal{L}_{P}(\theta_{k})^{T} G_{P}(\theta_{k},\xi_{k}) +\frac{\alpha^{2}L_{0}}{2} \left\lVert G_{P}(\theta_{k},\xi_{k})\right\rVert_{\ell_{2}}^{2} \\
      & \leqslant &-\alpha \left(\left\lVert \triangledown_{\theta}\mathcal{L}_{P}(\theta_{k})\right\rVert_{\ell_{2}}^{2}-C_{4} n^{-1}\left(\log n\right)^{d} \left\lVert \triangledown_{\theta}\mathcal{L}_{P}(\theta_{k})\right\rVert_{\ell_{2}}\right)\\
      & \mspace{21mu} &+\frac{\alpha^{2}L_{0}}{2}\left(\left\lVert \triangledown_{\theta}\mathcal{L}_{P}(\theta_{k})\right\rVert_{\ell_{2}}^{2}+C_{4}^{2} n^{-2}\left(\log n\right)^{2d}+2C_{4} n^{-1}\left(\log n\right)^{d} \left\lVert \triangledown_{\theta}\mathcal{L}_{P}(\theta_{k})\right\rVert_{\ell_{2}}\right) \\
      & \leqslant &\left(\frac{\alpha^{2}L_{0}}{2}-\alpha\right) \left\lVert \triangledown_{\theta}\mathcal{L}_{P}(\theta_{k})\right\rVert_{\ell_{2}}^{2}+\frac{\alpha^{2}L_{0}C_{4}^{2}}{2}  n^{-2}\left(\log n\right)^{2d} \\
      & & + \left(\alpha^{2}L_{0}+\alpha\right) C_{4} M_{P}  n^{-1}\left(\log n\right)^{d} \\
      & \leqslant & \left(\alpha^{2}L_{0}c-2\alpha c\right) \left(\mathcal{L}_{P}(\theta_{k})-\mathcal{L}_{P}(\theta^{*})\right)+\frac{\alpha^{2}L_{0}C_{4}^{2}}{2}  n^{-2}\left(\log n\right)^{2d} \\
      & &+ \left(\alpha^{2}L_{0}+\alpha\right) C_{4} M_{P}  n^{-1}\left(\log n\right)^{d}.
    \end{eqnarray*}
    Proceeding as in the proof of Theorem~\ref{thm35}, we denote
    \begin{equation*}
      r(n)\coloneqq \frac{1}{4c-2\alpha L_{0} c} \left(\alpha L_{0} C_{4}^{2}n^{-2}\left(\log n\right)^{2d}+\left(2\alpha L_{0}+2\right)C_{4} M_{P}n^{-1}\left(\log n\right)^{d}\right).
    \end{equation*}
    Then we have
    \begin{equation*}
      \mathcal{L}_{P}(\theta_{k+1})-\mathcal{L}_{P}(\theta^{*})-r(n) \leqslant \left(\alpha^{2}L_{0}c-2\alpha c+1\right) \left(\mathcal{L}_{P}(\theta_{k})-\mathcal{L}_{P}(\theta^{*})-r(n)\right).
    \end{equation*}
    Note that $0<\alpha < 2/L_{0}$ and $c<L_{0}$ indicate $0< \alpha^{2}L_{0} c-2 \alpha c+1<1$, which guarantees the convergence. Taking $k$ to infinity, we obtain
    \begin{equation*}
      \lim_{k \to \infty} \mathcal{L}_{P}(\theta_{k})-\mathcal{L}_{P}(\theta^{*})=r(n)=O\left(n^{-1}\left(\log n\right)^{d}\right).
    \end{equation*}
  \end{proof}

  Now we turn to the training error for solving the static Schr\"{o}dinger equation~(\ref{eq2}). Based on the similar ideas as above, we need to prove the Lipschitz continuity only. Recall the loss functional with respect to the static Schr\"{o}dinger equation~(\ref{eq2}) are
  \begin{equation*}
    \mathcal{L}_{S}(\theta)\coloneqq\int\nolimits_{\Omega}\left(\frac{1}{2}\left\lVert \triangledown_{x} v(x;\theta)\right\rVert_{\ell_{2}}^{2} +\frac{1}{2}V(x)\left\lvert v(x;\theta)\right\rvert^{2}-g(x)v(x;\theta)\right)\mathrm{d}x.
  \end{equation*}
  \begin{theorem}
    If the function $g$ is bounded in $\Omega$, then the gradient function of $\mathcal{L}_{S}(\theta)$ is Lipschitz continuous with Lipschitz constant $L_{7}>0$, i.e., $\forall \theta, \overline{\theta} \in \Theta$,
    \begin{equation*}
      \left\lVert \triangledown_{\theta}\mathcal{L}_{S}(\theta)-\triangledown_{\theta}\mathcal{L}_{S}(\overline{\theta})\right\rVert_{\ell_{2}}\leqslant L_{7}\left\lVert \theta - \overline{\theta} \right\rVert_{\ell_{2}}.
    \end{equation*}
  \end{theorem}
  \begin{proof}
    Using Lemmas~\ref{lm32} and~\ref{lm35}, we have
    \begin{eqnarray*}
      &\mspace{21mu}& \left\lVert\triangledown_{\theta}\left(V(x)\left\lvert v(x;\theta)\right\rvert^{2}\right)-\triangledown_{\theta}\left(V(x)\left\lvert v(x;\overline{\theta})\right\rvert^{2}\right)\right\rVert_{\ell_{2}}              \\
      &=&\left\lVert V(x)\left(2v(x;\theta)\triangledown_{\theta}v(x;\theta)-2v(x;\overline{\theta})\triangledown_{\theta}v(x;\overline{\theta})\right)\right\rVert_{\ell_{2}} \\
      &\leqslant& 2V_{max} \left\lVert v(x;\theta)\triangledown_{\theta}v(x;\theta)-v(x;\overline{\theta})\triangledown_{\theta}v(x;\overline{\theta})\right\rVert_{\ell_{2}}
      \end{eqnarray*}
      and
      \begin{eqnarray*}
      &\mspace{21mu}& \left\lVert v(x;\theta)\triangledown_{\theta}v(x;\theta)-v(x;\overline{\theta})\triangledown_{\theta}v(x;\overline{\theta})\right\rVert_{\ell_{2}} \\
      &\leqslant &\left\lVert v(x;\theta)\triangledown_{\theta}v(x;\theta)-v(x;\theta)\triangledown_{\theta}v(x;\overline{\theta})\right\rVert_{\ell_{2}}+ \left\lVert v(x;\theta)\triangledown_{\theta}v(x;\overline{\theta})-v(x;\overline{\theta})\triangledown_{\theta}v(x;\overline{\theta})\right\rVert_{\ell_{2}}   \\
      &\leqslant&  B_{1} \left\lVert \triangledown_{\theta}v(x;\theta)- \triangledown_{\theta}v(x;\overline{\theta})\right\rVert_{\ell_{2}}+ B_{3} \left\lvert v(x;\theta)-v(x;\overline{\theta})\right\rvert    \\
      &\leqslant& B_{1} L_{3} \left\lVert\theta - \overline{\theta} \right\rVert_{\ell_{2}}+B_{3}L_{1} \left\lVert\theta - \overline{\theta} \right\rVert_{\ell_{2}}.   \\
    \end{eqnarray*}
    Hence, $\triangledown_{\theta}(V(x)\left\lvert v(x;\theta)\right\rvert^{2})$ is Lipschitz continuous with respect to $\theta$ for all $x \in \Omega$. The rest of the proof is quite similar to that given earlier for the Lipschitz continuity of $\triangledown_{\theta}\mathcal{L}_{P}(\theta)$ and is omitted.
  \end{proof}
  Translating the previous assumptions to the case of the static Schr\"{o}dinger equation~(\ref{eq2}), we  have the following theorem on the training error.
  \begin{theorem}\label{thm38}
    \begin{enumerate}
      \item[(i)] For solving the static Schr\"{o}dinger equation~(\ref{eq2}), under the same assumptions given in Lemma~\ref{lm36} replacing $\mathcal{L}_{P}(\theta)$ with $\mathcal{L}_{S}(\theta)$, if the iteration stepsize $\alpha$ is a constant satisfying $0<\alpha < 2/L_{7}$, then the training error of DRM satisfies
      \begin{equation*}
        \Delta \mathcal{L}_{MCtra} = O\left(n^{-1}\right).
      \end{equation*}
      \item[(ii)] For solving the static Schr\"{o}dinger equation~(\ref{eq2}), under the same assumption given in Lemma~\ref{lm36} replacing $\mathcal{L}_{P}(\theta)$ with $\mathcal{L}_{S}(\theta)$, if the iteration stepsize $\alpha$ is a constant satisfying $0<\alpha < 2/L_{7}$ and $V\in C^{d}(\Omega)$, $g\in C^{d}(\Omega)$, then the training error of DRM-QMC satisfies
      \begin{equation*}
        \Delta \mathcal{L}_{QMCtra}=O\left(n^{-1}(\log n)^{d}\right).
      \end{equation*}
    \end{enumerate}
  \end{theorem}
  This theorem can be proved in the same way as before, so the proof will not be reproduced here.
  \begin{rem}\label{remark33}
    If the strongly convex constant $c$ satisfies $c=L_{0}$, then we establish the convergence order in two cases of the fixed stepsize $\alpha$.
    \begin{enumerate}
      \item[(i)] If $\alpha\neq 1/L_{0}$, then the results in Theorems~\ref{thm35},~\ref{thm36} and~\ref{thm38} also hold.
      \item[(ii)] If $\alpha= 1/L_{0}$, then the training errors obtained in Theorems~\ref{thm35},~\ref{thm36} and~\ref{thm38} become upper bounds on the corresponding training error, namely the $"="$ becomes $"\leqslant"$.
    \end{enumerate}
  \end{rem}
  \begin{rem}
    It should be acknowledged that the strong convexity assumption of the loss functional in Lemma~\ref{lm36} does not always hold and it is difficult to verify whether the strong convexity assumption holds. Such assumption is usually essential in gradient-based method. For the non-convex case, we can consider the nonconvex optimization methods in~\cite{Chen1,Fehrman} and we leave it for future research. However, it is convincing that the training error of DRM-QMC must be no worse than that of DRM in all cases due to the use of low discrepancy sequence.
  \end{rem}
  
  \subsection{Comparison of the total errors}
  We have established the generalization error bound and the training error with respect to different sampling strategies. Combining with Theorem~\ref{thm21} and the triangle inequality, we obtain the upper bounds on the difference between the limit of the algorithm output and the unique weak solution. Under the settings of the previous theorems, we can summarize the results of the generalization error and the training error in the following table.

  \begin{table}[htbp]
    \centering
    \caption{\centering{Theoretical results.}}
    \label{tab1}
  \begin{tabular}{|l|l|l|}
    \hline
    \multicolumn{1}{|c|}{}
    &\multicolumn{1}{|c|}{generalization error}
    &\multicolumn{1}{|c|}{training error}\\\hline
    \makecell{DRM}  &$\leqslant O\left(n^{-1/2}(\log n)^{1/2}\right)$  &$=O\left(n^{-1}\right)$   \\\hline
    \makecell{DRM-QMC}  &$\leqslant O\left(n^{-1}(\log n)^{d}\right)$  &$= O\left(n^{-1}(\log n)^{d}\right)$ \\\hline
  \end{tabular}
  \end{table}

  Since the approximation error is independent of the sampling strategy, it is not presented in Table~\ref{tab1}. In conclusion, DRM-QMC is asymptotically better than DRM in terms of the generalization error bound, and the order of the training error of DRM-QMC is asymptotically equal to that of DRM. To compare the total error with respect to different sampling strategies, we consider the dominant term of the total error. For DRM, it is $O\left(n^{-1/2}(\log n)^{1/2}\right)$. For DRM-QMC, it is $O\left(n^{-1}(\log n)^{d}\right)$. Hence, for DRM, the generalization error dominates the total error regardless of the approximation error, while the generalization error measures the error incurred by the quadrature method. Therefore, modifying the quadrature method, that is, replacing MC with QMC, can bring great accuracy improvements. As for the lower bounds, Lu et al.~\cite{Lu1} establish a lower bound for the case of the static Schr\"{o}dinger equation with the Dirichlet boundary condition. That is, for a given $s\in \mathbb{N}^{+}$, 
  \begin{equation*}
      \mathop{\inf}_{v \in \mathcal{F}}\mathop{\sup}_{u^{*} \in H^{s}(\Omega)}\mathbb{E}\left\lVert v-u^{*} \right\lVert_{H^{1}(\Omega)}^{2}\gtrsim n^{-\frac{2s-2}{d+2s-4}}.
  \end{equation*}
  If $s=1$, the error will not decrease below a certain constant independent of the mini-batch size in the training process. This phenomenon will be demonstrated in our numerical experiments.

  From the theoretical results, DRM-QMC achieves an asymptotically smaller total error bound than DRM with or without the assumption of strong convexity, which means DRM-QMC can be expected to be more accurate than DRM. Moreover, to verify the better efficiency of DRM-QMC over DRM, we perform some numerical experiments in Section~\ref{sec4}.

\section{Numerical Experiments} \label{sec4}
  In our numerical experiments, we pay attention to the convergence rate and stability of the algorithm which may benefit from QMC methods. We apply the DRM combined with different sampling strategies to find the numerical solutions of two problems, which correspond to the two types of PDE problems studied. To avoid the vanishing gradient problem~\cite{He}, we add the residual to the neural network. Specifically, we give the mathematical form of the deep neural network
  \begin{equation*}
    f_{1}(s)=\sigma \circ T_{2} \circ \sigma \circ T_{1}(s)+s, \quad \quad f_{2}(s)=\sigma \circ T_{4} \circ \sigma \circ T_{3}(s)+s,
  \end{equation*}
  \begin{equation*}
    \Downarrow
  \end{equation*}
  \begin{equation*}
    v(x;\theta)= T_{out} \circ f_{2} \circ f_{1} \circ T_{in}(x),
  \end{equation*}
  where $T_{i} : \mathbb{R} ^{4}\rightarrow \mathbb{R}^{4}, s \mapsto A_{i}s_{i}+B_{i}, \; i=1, 2, 3, 4$, $T_{in}: \mathbb{R}^{d}\to \mathbb{R}^{4}, x \mapsto  A_{in}x+B_{in}$, and $T_{out}: \mathbb{R}^{4}\to \mathbb{R}, s \mapsto  A_{out}s+B_{out}$.
  
  To illustrate the difference between the convergence rates of DRM and DRM-QMC, we compare the relative $L_{2}$ errors, i.e.,
  \begin{equation*}
    error_{L_{2}} = \sqrt{\frac{\int_{\Omega}\left(v(x;\theta)-u^{*}(x)\right)^{2}\mathrm{d}x}{\int_{\Omega}u^{*}(x)^{2}\mathrm{d}x} }.
  \end{equation*}
  For a gradient-based optimization algorithm, smaller variance of the gradient estimator means the algorithm is more stable. In this paper, the gradient we considered is vector-valued, so we use the trace of sample covariance matrix of the gradient estimator to measure the stability of the algorithm. Since QMC methods use deterministic points as sample points, we employ randomized quasi-Monte Carlo (RQMC) points to compute the sample covariance matrix, called randomized DRM-QMC. By the central limit theorem, we know that the variance of the MC estimator is $O\left(n^{-1}\right)$, while the variance of the RQMC estimator is $O\left(n^{-2+\epsilon}\right)$ for any $\epsilon >0$ or even $O\left(n^{-k}\right)$, $k>2$, under some conditions~\cite{Owen}.  
  In numerical experiments, we compare the convergence rates of DRM and DRM-QMC with the mini-batch size equaling to 32, 128 or 512, respectively, and the number of iterations is fixed to 10000. Furthermore, we perform 16 repetitions to DRM and randomized DRM-QMC in each step and compare the traces of the sample covariance matrices.  

  First, we consider the following Neumann problem for the Poisson equation.
  \begin{example}\label{example41}
  \begin{equation} \label{eq8}
    \left\{
    \begin{aligned}
      \Delta u&= 2\sum_{k = 1}^{20}\left(x_{k}^{2}-x_{k}\right)^{2}+\sum_{k = 1}^{20}\left(4x_{k}-2\right)\sum_{k = 1}^{20}\left(\frac{x_{k}^{3}}{3}-\frac{x_{k}^{2}}{2}\right), \ &\rm{in} \quad \Omega, \\
      \frac{\partial u}{\partial \boldsymbol{n}}&=0, \ &\rm{on} \ \partial \Omega,
    \end{aligned}
    \right.
  \end{equation}
  where $\Omega=[0,1]^{20}$.
\end{example}
  By routine computation, the unique weak solution satisfying $\int_{\Omega}u^{*}(x)\mathrm{d}x=0$ of~(\ref{eq8}) is
  \begin{equation*}
    u^{*}(x)=\left(\sum_{k=1}^{20}\left(\frac{x_{k}^{3}}{3}-\frac{x_{k}^{2}}{2}\right)\right)^{2}-\frac{717}{252}.
  \end{equation*}
  Based on the idea of DRM, the corresponding variational problem is
  \begin{equation*}
    \mathop{\arg\min}_{\theta \in \Theta}\int_{\Omega}\left(\frac{1}{2}\left\lVert \triangledown_{x} v(x;\theta)\right\rVert_{\ell_{2}}^{2} -f(x) v(x;\theta)\right)\mathrm{d}x+\frac{1}{2}\left( \int\nolimits_{\Omega}v(x;\theta)\mathrm{d}x\right)^{2},
  \end{equation*}
  where 
  \begin{equation*}
    f(x)=-\left(2\sum_{k = 1}^{20}\left(x_{k}^{2}-x_{k}\right)^{2}+\sum_{k = 1}^{20}\left(4x_{k}-2\right)\sum_{k = 1}^{20}\left(\frac{x_{k}^{3}}{3}-\frac{x_{k}^{2}}{2}\right)\right).  
  \end{equation*}

  \begin{figure}[htbp]
  \centering
  \includegraphics[height=4.5cm,width=13cm]{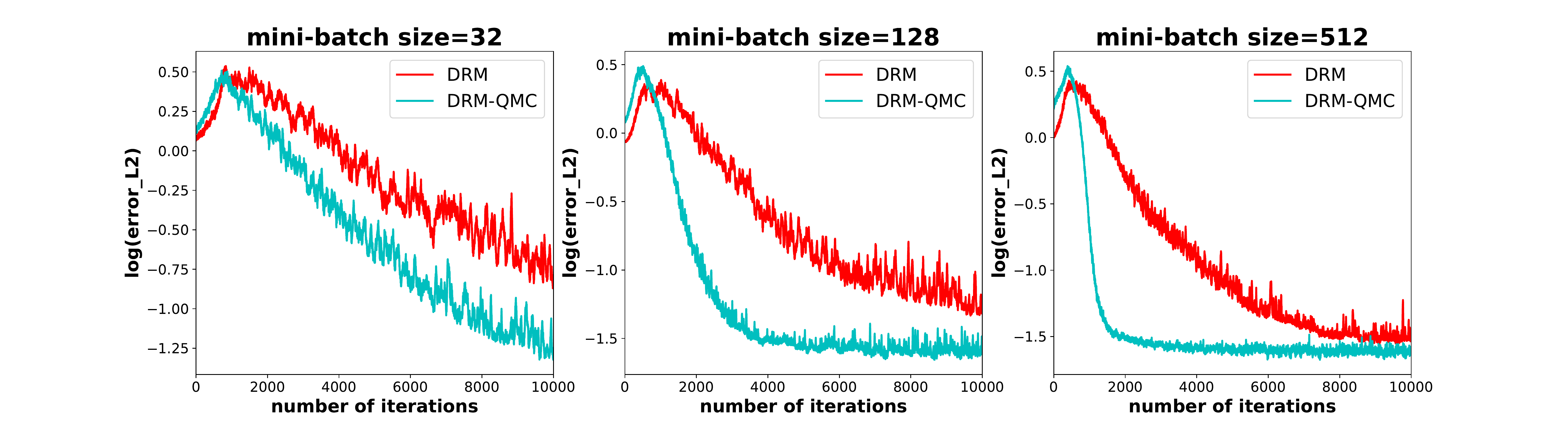}
  \caption{Training processes with respect to different sampling strategies with the same mini-batch size for solving the Poisson equation~(\ref{eq8})}\label{fig1}
\end{figure}
\begin{figure}[htbp]
  \centering
  \includegraphics[height=4.5cm,width=13cm]{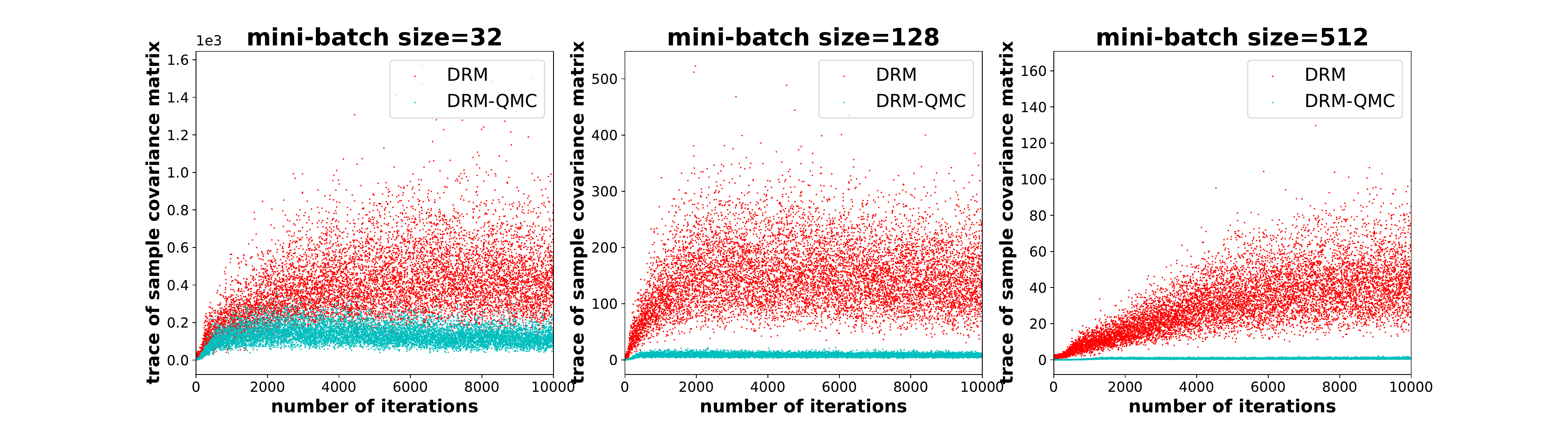}
  \caption{Traces of sample covariance matrices with respect to different sampling strategies with the same mini-batch size for solving the Poisson equation~(\ref{eq8}). Notice that the scaling is different.}\label{fig2}
\end{figure}

  Figure~\ref{fig1} shows the convergence rates of DRM and DRM-QMC with the same mini-batch size. When the mini-batch size equals to 32, neither DRM nor DRM-QMC converges after 10000 iterations, while the minimal error that DRM-QMC can achieve is about 60$\%$ of that DRM can achieve in the training processes. When the mini-batch size equals to 128, DRM-QMC converges after about 4000 iterations, while DRM still does not converge after 10000 iterations. When the mini-batch size equals to 512, DRM-QMC converges after about 2000 iterations and DRM converges after about 8000 iterations. 
  
  Figure~\ref{fig2} presents the stability of DRM and DRM-QMC. The variances of the gradient estimators in randomized DRM-QMC are much smaller than those in DRM. In particular, when the mini-batch size equals to 128 or 512, the variances of the gradient estimators in randomized DRM-QMC are close to 0, while the variances of the gradient estimators in DRM range from dozens to hundreds.
  
  Next, we consider the following Neumann problem for the static Schr\"{o}dinger equation.
  \begin{example}\label{example42}
  \begin{equation} \label{eq9}
    \left\{
    \begin{aligned}
      -\Delta u+\pi^{2}u &=2\pi^{2}\sum_{k = 1}^{20}\cos (\pi x_{k}), \ &\rm{in} \quad \Omega, \\
      \frac{\partial u}{\partial \boldsymbol{n}}&=0, \ &\rm{on} \ \partial \Omega,
    \end{aligned}
    \right.
  \end{equation}
  where $\Omega=[0,1]^{20}$.
\end{example}

  Similarly, we know that the unique weak solution is
  \begin{equation*}
    u^{*}(x)=\sum_{k = 1}^{20}\cos (\pi x_{k})
  \end{equation*}
  and the corresponding variational problem is
  \begin{equation*}
    \mathop{\arg\min}_{\theta\in \Theta}\int_{\Omega}\left(\frac{1}{2}\left\lVert \triangledown_{x} v(x;\theta)\right\rVert_{\ell_{2}}^{2}+\frac{\pi^{2}}{2} \left\lvert v(x;\theta)\right\rvert^{2}-2\pi^{2}\sum_{k = 1}^{20}\cos (\pi x_{k}) v(x;\theta)\right)\mathrm{d}x.
  \end{equation*}

\begin{figure}[htbp]
  \centering
  \includegraphics[height=4.5cm,width=13cm]{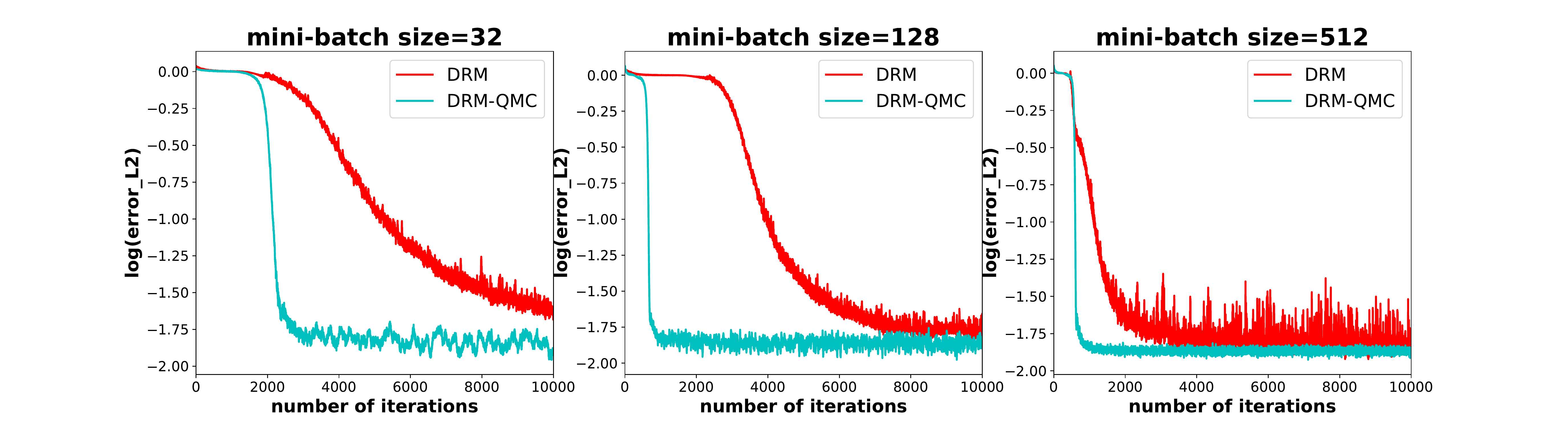}
  \caption{Training processes with respect to different sampling strategies with the same mini-batch size for solving the static Schr\"{o}dinger equation~(\ref{eq9})}\label{fig3}
\end{figure}

\begin{figure}[htbp]
  \centering
  \includegraphics[height=4.5cm,width=13cm]{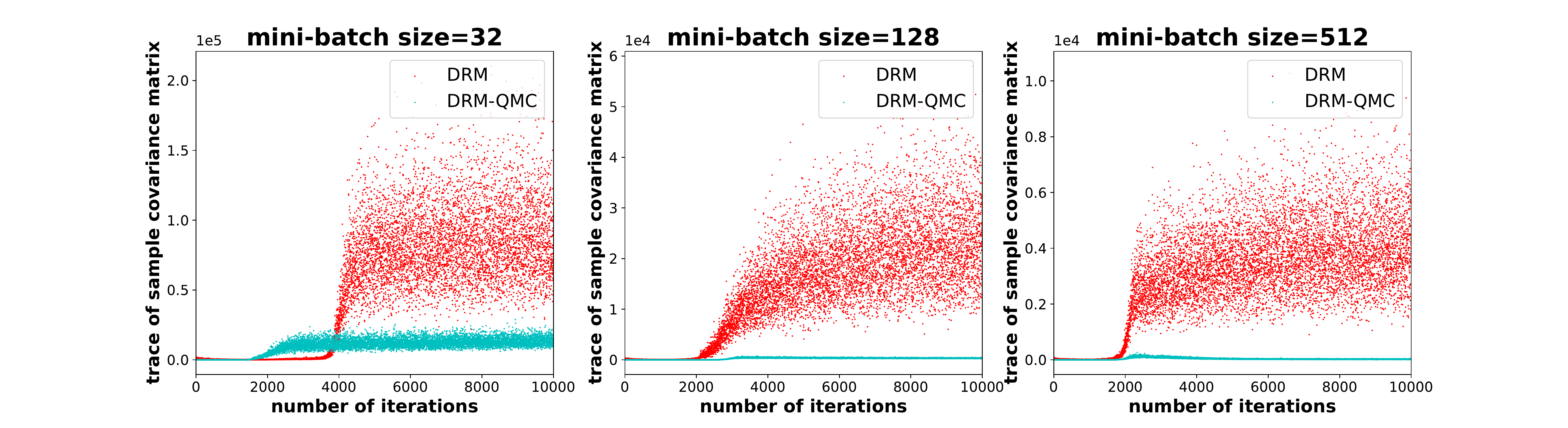}
  \caption{Traces of sample covariance matrices with respect to different sampling strategies with the same mini-batch size for solving the static Schr\"{o}dinger equation~(\ref{eq9}). Notice that the scaling is different.}\label{fig4}
\end{figure}
  
  Figure~\ref{fig3} compares the convergence rates of DRM and DRM-QMC with the same mini-batch size. When the mini-batch size equals to 32, DRM-QMC converges after about 3000 iterations, while DRM does not converge after 10000 iterations. Moreover, DRM-QMC converges after about 1000 iterations when the mini-batch size equals to 128 or 512, while DRM converges after about 8000 iterations when the mini-batch size equals to 128 and after about 4000 iterations when the mini-batch size equals to 512. 
  
  Figure~\ref{fig4} shows the traces of sample covariance matrices with respect to DRM and randomized DRM-QMC. The variances of the gradient estimators in randomized DRM-QMC are closed to 0, compared to the variances of the gradient estimators in DRM which are thousands. 
  
  To be more precise, we compute the ratios of the traces of sample covariance matrices of DRM to those of randomized DRM-QMC in each step and average these ratios. The results for solving the Poisson equation~(\ref{eq8}) and the static Schr\"{o}dinger equation~(\ref{eq9}) are listed in Table~\ref{tab2}.
  \begin{table}[htbp]
    \centering
    \caption{The average ratio of the traces of sample covariance matrices of DRM to that of randomized DRM-QMC.}
    \label{tab2}
  \begin{tabular}{|l|l|l|l|}
    \hline
    \multicolumn{1}{|c|}{}
    &\multicolumn{1}{|c|}{$n=32$}
    &\multicolumn{1}{|c|}{$n=128$}
    &\multicolumn{1}{|c|}{$n=512$}\\\hline
    \makecell{Poisson equation~(\ref{eq8})} &\makecell{3.3}  &\makecell{19.3} &\makecell{53.6}  \\\hline
    \makecell{Static Schr\"{o}dinger equation~(\ref{eq9})}  &\makecell{12.4}  &\makecell{142.3} &\makecell{1052.0} \\\hline
  \end{tabular}
  \end{table}

  By summarizing the results of the numerical experiments, we observe that the proposed algorithm always performs better than the standard algorithm. Specifically, we conclude three advantages of DRM-QMC.
  \begin{enumerate}
    \item[(i)] DRM-QMC converges faster than DRM in all cases we considered. This illustrates the superiority of the proposed algorithm in terms of the convergence rate.
    \item[(ii)] When the training process is stable, DRM-QMC usually achieves smaller error than DRM does. This is consistent with the theoretical error analysis in Section~\ref{sec3}.
    \item[(iii)] The trace of sample covariance matrix of randomized DRM-QMC can be reduced by factors ranging from 3.3 to 1052.0 over DRM. Furthermore, as the mini-batch size increases, the variances of the gradient estimators in randomized DRM-QMC decrease faster than those in DRM, which means DRM-QMC is much more stable than DRM during the training process and the increase in mini-batch size affects the stability to a greater extent for DRM-QMC than for DRM.
  \end{enumerate}
  Chen et al.~\cite{Chen} have shown that DRM-QMC performs better than DRM for solving elliptic PDEs equipped with the Dirichlet boundary condition. We also apply DRM and DRM-QMC to solve some other PDE problems and different constructions of the deep neural network. The results (though not presented here) show that DRM-QMC usually performs better than DRM, especially when the DRM is hard to converge.

  \section{Conclusion}\label{sec5}
  In this paper, we combined DRM with QMC methods and analyzed the effect of different sampling strategies on the DRM. From the aspects of both theoretical results and numerical experiments, we compared the accuracy and efficiency of the proposed algorithm and the standard algorithm.

  Theoretically, the error of using the deep learning algorithms to solve PDEs is decomposed into the generalization error, the approximation error and the training error. Analyzing two types of errors related to the sampling strategy rigorously and summarizing the results, we obtained that DRM-QMC is asymptotically better than DRM in terms of the error bound. Hence, DRM-QMC may give an output that is closer to the exact solution of PDEs than DRM under some conditions, which means DRM-QMC is more accurate than DRM.

  From the results of the numerical experiments, DRM-QMC converges faster and is more stable than DRM. For the same accuracy requirement, DRM-QMC requires fewer iteration steps and fewer sample points in the training process, which means the computation cost will be greatly reduced in practical applications. For small mini-batch size, DRM-QMC can keep the convergence rate at a satisfactory level better than DRM. For large mini-batch size, DRM-QMC presents greater stability improvement than DRM. In general, DRM-QMC performs better than DRM for both small and large mini-batch size.

  DRM and DRM-QMC can be applied to solve other second-order elliptic equations. Moreover, QMC methods can be applied to other deep learning algorithms for solving PDE problems, for example, DGM and PINNs. The study of QMC-based deep learning algorithm combined with nonconvex optimization methods is also an interesting topic. The relevant theoretical error analysis is left as future research.

%%%% Acknowledgments %%%%%%%%
\section*{Acknowledgments}
This work is supported by the National Natural Science Foundation of China through grant 72071119.

\end{document}